\documentclass[a4paper,10pt]{amsart}
\usepackage{amsmath,amssymb}
\usepackage{amscd}
\newtheorem{thm}{Theorem}[section]
\newtheorem{prop}[thm]{Proposition}
\newtheorem{lemma}[thm]{Lemma}

\newtheorem{exam}[thm]{Example}
\theoremstyle{remark}
\newtheorem{remark}[thm]{Remark}
\newcommand{\id}{{\rm{id}}}
\newcommand{\Ad}{{\rm{Ad}}}

\newcommand{\BC}{\mathbf C}
\newcommand{\BZ}{\mathbf Z}

\newcommand{\la}{\langle}
\newcommand{\ra}{\rangle}
\newcommand{\Ind}{{\rm{Ind}}}

\newtheorem{Def}{Definition}[section]

\title{Equivalence bundles over a finite group and strong Morita equivalence for unital
inclusions of unital $C^*$-algebras}
\author{Kazunori Kodaka}
\address{Department of Mathematical Sciences, Faculty of Science, Ryukyu
\endgraf
University, Nishihara-cho, Okinawa, 903-0213, Japan}

\address{\sl{E-mail address}: \rm{kodaka@math.u-ryukyu.ac.jp}}

\keywords{$C^*$-algebraic bundles, equivalence bundles, inclusions of $C^*$-algebras,
strong Morita equivalence}

\subjclass[2010]{Primary 46L05, Secondary 46L08}

\begin{document}
\maketitle
\begin{abstract}
Let $\mathcal{A}=\{A_t \}_{t\in G}$ and $\mathcal{B}=\{B_t \}_{t\in G}$ be $C^*$-algebraic
bundles over a finite group $G$. Let $C=\oplus_{t\in G}A_t$ and
$D=\oplus_{t\in G}B_t$. Also, let $A=A_e$ and $B=B_e$,
where $e$ is the unit element in $G$.
We suppose that $C$ and $D$ are unital and $A$ and $B$
have the unit elements in $C$ and $D$, respectively. In this paper, we shall show that if there is an
equivalence $\mathcal{A}-\mathcal{B}$-bundle over $G$ with some properties,
then the unital inclusions of unital $C^*$-algebras $A\subset C$ and $B\subset D$
induced by $\mathcal{A}$ and $\mathcal{B}$ are strongly Morita equivalent. 
Also, we suppose that $\mathcal{A}$ and $\mathcal{B}$ are saturated and that
$A' \cap C=\BC 1$.
We shall show that if $A\subset C$ and $B\subset D$
are strongly Morita equivalent,
then there are an automorphism $f$ of $G$ and an equivalence bundle
$\mathcal{A}-\mathcal{B}^f $-bundle over $G$ with the above properties,
where $\mathcal{B}^f$ is the $C^*$-algebraic bundle induced by $\mathcal{B}$ and $f$,
which is defined by $\mathcal{B}^f =\{B_{f(t)}\}_{t\in G}$.
Furthermore, we shall give an application.
\end{abstract}

\section{Intrtoduction}\label{sec:intro} Let $\mathcal{A}=\{A_t \}_{t\in G}$ be a $C^*$-algebraic bundle
over a finite group $G$. Let $C=\oplus_{t\in G}A_t$ and $A_e =A$, where $e$
is the unit element in $G$. We suppose that $C$ is unital and
that $A$ has the unit element in $C$. Then we obtain a unital inclusion of unital $C^*$-algebras,
$A\subset C$. We call it the unital inclusion of unital $C^*$-algebras
\sl
induced by a $C^*$-algebraic bundle $\mathcal{A}=\{A_t \}_{t\in G}$.
\rm
Let $E^A$ be the canonical conditional expectation from $C$ onto $A$ defined by
$$
E^A (x)=x_e
$$
for all $x=\sum_{t\in G}x_i \in C$.

\begin{Def}\label{def:saturated}Let $\mathcal{A}=\{A_t \}_{t\in G}$
be a $C^*$-algebraic bundle over a finite group $G$. We say that
$\mathcal{A}$ is
\sl
saturated
\rm
if $\overline{A_t A_t^*}=A$ for all $t\in G$.
\end{Def}

Since $A$ is unital, in our case we do not need to take the closure in Definition \ref{def:saturated}.
If $\mathcal{A}$ is saturated, by \cite [Corollary 3.2]{KT7:characterization}, $E^A$ is of
index-finite type and its Watatani index, $\Ind_W (E^A )=|G|$, where $|G|$ is the order of $G$.

Let $\mathcal{B}=\{B_t \}_{t\in G}$ be another $C^*$-algebraic bundle over $G$.
Let $D=\oplus_{t\in G}B_t$ and $B=B_e$. Also, we suppose that $\mathcal{B}$ has
the same conditions as $\mathcal{A}$. Let $B\subset D$ be the unital inclusion of unital $C^*$-algebras
induced by $\mathcal{B}$.
\par
Let $\mathcal{X}=\{X_t \}_{t\in G}$ be an $\mathcal{A}-\mathcal{B}$-equivalence bundle
defined by Abadie and Ferraro \cite [Definition 2.2]{AF:bundle}. Moreover, we suppose that
$$
{}_{C} \la X_t \, , \, X_s \ra=A_{ts^{-1}} \, , \quad
\la X_t \, , \, X_s \ra_{D}=B_{t^{-1}s}
$$
for any $t, s\in G$, where ${}_{C}\la X_t \, , \, X_s \ra $ means the linear span of the set
$$
\{{}_{C} \la x \, , \, y \ra\in A_{ts^{-1}} \, | \, x\in X_t , \, y\in X_s \}
$$
and $\la X_t \, , \, X_s \ra_{D} $ means the linear span of the similar set to the above.
The above two properties are stronger than the properties (7R) and (7L) in \cite [Definition 2.1]{AF:bundle}.
\par
In the present paper, we shall show that if there is an $\mathcal{A}-\mathcal{B}$-equivalence bundle
$\mathcal{X}=\{X_t \}_{t\in G}$ such that ${}_{C}\la X_t \, , \, X_s \ra=A_{ts^{-1}}$ and
$\la X_t \, , \, X_s \ra_{D}=B_{t^{-1}s}$ for any $t, s\in G$, then the unital inclusions of
unital $C^*$-algebras $A\subset C$ and $B\subset D$ induced by $\mathcal{A}$ and $\mathcal{B}$
are strongly Morita equivalent.
Also, we suppose that $\mathcal{A}$ and $\mathcal{B}$ are saturated and that $A' \cap C=\BC 1$.
We shall show that
if $A\subset C$ and $B\subset D$ are strongly
Morita equivalent, then there are an automorphism $f$ of $G$ and
an $\mathcal{A}-\mathcal{B}^f $-equivalence bundle $\mathcal{X}=\{X_t \}_{t\in G}$
such that ${}_{C}\la X_t \, , \, X_s \ra=A_{ts^{-1}}$ and
$\la X_t \, , \, X_s \ra_{D}=B_{f(t^{-1}s)}$ for any $t, s\in G$,
where $\mathcal{B}^f $ is the $C^*$-algebraic bundle induced by
$\mathcal{B}=\{B_t \}_{t\in G}$ and $f$, which is defined by $\mathcal{B}^f =\{B_{f(t)}\}_{t\in G}$.
\par
Let $A$ and $B$ be unital $C^*$-algebras and $X$ an $A-B$-equivalence
bimodule. Then we denote its left $A$-action and right $B$-action on $X$ by
$a\cdot x$ and $x\cdot b$ for any $a\in A$, $b\in B$ and $x\in X$, respectively.
Also, we mean by the words ``Hilbert $C^*$-bimodules" Hilbert $C^*$-bimodules in the
sense of Brown, Mingo and Shen \cite {BMS:quasi}.

\section{Equivalence bundles over a finite group}\label{sec:bundle} Let $\mathcal{A}=\{A_t \}_{t\in G}$ and
$\mathcal{B}=\{B_t \}_{t\in G}$ be $C^*$-algebraic bundles over a finite group $G$.
Let $e$ be the unit element in $G$.
Let $C=\oplus_{t\in G}A_t $, $D=\oplus_{t\in G}B_t$ and $A=A_e$, $B=B_e$.
We suppose that $C$ and $D$ are unital and that $A$ and $B$ have the unit elements
in $C$ and $D$, respectively. Let $\mathcal{X}=\{X_t \}_{t\in G}$
be an $\mathcal{A}-\mathcal{B}$-equivalence bundle over $G$ such that 
$$
{}_{C} \la X_t \, , \, X_s \ra=A_{ts^{-1}} , \quad
\la X_t \, , \, X_s \ra_{D}=B_{t^{-1}s}
$$
for any $t, s\in G$.
Let $Y=\oplus_{t\in G}X_t$
and $X=X_e$. Then $Y$ is a $C-D$-equivalence bimodule by
Abadie and Ferraro \cite [Definitions 2.1,  2.2 ]{AF:bundle},
Also, $X$ is an $A-B$-equivalence bimodule since
${}_{C} \la X \, , \, X \ra=A$ and $\la X \, , \, X \ra_{D}=B$.

\begin{prop}\label{prop:easy} Let $\mathcal{A}=\{A_t \}_{t\in G}$ and $\mathcal{B}=\{B_t \}_{t\in G}$ be
$C^*$-algebraic bundles over a finite group $G$.
Let $C=\oplus_{t\in G}A_t$ and $D=\oplus_{t\in G}B_t$ Also, let $A=A_e$ and $B=B_e$,
where $e$ is the unit element in $G$. We suppose that $C$ and $D$ are unital and that
$A$ and $B$ have the unit elements in $C$ and $D$, respectively. Also,
we suppose that there is an $\mathcal{A}-\mathcal{B}$-
equivalence bundle $\mathcal{X}=\{X_t \}_{t\in G}$ over $G$ such that
$$
{}_{C} \la X_t \, , \, X_s \ra=A_{ts^{-1}} \, , \quad \la X_t \, , \, X_s \ra_{D}=B_{t^{-1}s}
$$
for any $t, s\in G$. Then the unital inclusions of unital $C^*$-algebras $A\subset C$ and $B\subset D$ are
strongly Morita equivalent.
\end{prop}
\begin{proof}Let $Y=\oplus_{t\in G}X_t$ and $X=X_e$. By the above discussions and \cite [Definition 2.1]{KT4:morita},
we have only to show that
$$
{}_C \la Y \, , \, X \ra=C \, , \quad  \la Y \, , \, X \ra_D=D .
$$
Let $x\in X$ and $y=\sum_{t\in G}y_t \in Y$, where $y_t\in X_t$ for any $t\in G$. Then
$$
{}_C \la y, x \ra =\sum_{t\in G} {}_{C} \la y_t \, , x \ra \, ,\quad
\la y, x \ra_D =\sum_{t\in G}\la y_t \, , x \ra_{D} .
$$
We note that ${}_{C} \la y_t \, , x \ra \in A_t$ and $\la y_t \, , x \ra_{D}\in B_t$
for any $t\in G$. Since ${}_{D} \la X_t \, , \, X_s \ra=A_{ts^{-1}}$ and
$\la X_t \, , \, X_s \ra_{D}=B_{t^{-1}s}$ for any $t, s\in G$, by the above computations,
we can see that
$$
{}_C \la Y \, , \, X \ra=C \, , \quad \la Y \, , \, X \ra_D=D .
$$
Therefore, we obtain the conclusion.
\end{proof}

Next, we shall give an example of an equivalence bundle $\mathcal{X}=\{X_t \}_{t\in G}$
over $G$ satisfying the above properties. In order to do this, we prepare a lemma.
Let $\mathcal{A}=\{A_t \}_{t\in G}$ and $\mathcal{B}=\{B_t \}_{t\in G}$ be as above.
Let $\mathcal{X}=\{X_t \}_{t\in G}$ be a complex Banach bundle over $G$ with the
maps defined by
\begin{align*}
& (y, d)\in Y\times D\mapsto y\cdot d\in Y , \quad (y, z)\in Y\times Y \mapsto \la y, z \ra_D \in D \\
& (c, y)\in C\times Y \mapsto c\cdot y \in Y , \quad (y, z)\in Y\times Y\mapsto {}_C \la y , z \ra\in C ,
\end{align*}
where $Y=\oplus_{t\in G}X_t$.

\begin{lemma}\label{lem:equivalence} With the above notation, we suppose that by the above maps,
$Y$ is a $C-D$-equivalence bimodule satisfying that
$$
{}_C \la X_t , X_s \ra =A_{ts^{-1}} , \quad \la X_t , X_s \ra_{D}=B_{t^{-1}s}
$$
for any $t, s\in G$. If $\mathcal{X}$ satisfies Conditions $(1R)$-$(3R)$ and
$(1L)$-$(3L)$ in \cite [Definition 2.1]{AF:bundle}, then $\mathcal{X}$ is
an $\mathcal{A}-\mathcal{B}$-equivalence bundle.
 \end{lemma}
 \begin{proof} Since $Y$ is a $C-D$-equivalence bimodule, $\mathcal{X}$ has Conditions (4R)-(6R)
 and (4L)-(6L) in \cite [Definiton 2.1]{AF:bundle} except that $X_t$ is complete with the norms
$||\la -, - \ra_D ||^{\frac{1}{2}}=||{}_C \la - , - \ra ||^{\frac{1}{2}}$ for any $t\in G$. But we know that
if $Y$ is complete with two different norms, then the two norms are equivalent. Hence $X_t$ is complete with
the norms $||\la -, - \ra_D ||^{\frac{1}{2}}=||{}_C \la - , - \ra ||^{\frac{1}{2}}$ for any $t\in G$. Furthermore, since
$$
{}_C \la X_t , X_s \ra =A_{ts^{-1}} , \quad \la X_t , X_s \ra_{D}=B_{t^{-1}s}
$$
for any $t, s\in G$, $\mathcal{X}$ has Conditions (7R) and (7L) in \cite [Definiton 2.1]{AF:bundle}.
Therefore, we obtain the conclusion.
 \end{proof}

We give an example of an $\mathcal{A}-\mathcal{B}$-equivalence bundle $\mathcal{X}=\{X_t \}_{t\in G}$
such that
$$
{}_{C} \la X_t \, , \, X_s \ra=A_{ts^{-1}} \, , \quad \la X_t \, , \, X_s \ra_{D}=B_{t^{-1}s}
$$
for any $t, s\in G$. 

\begin{exam}\label{exam:action}
\rm
Let $G$ be a finite group. Let $\alpha$ be an
action of $G$ on a unital $C^*$-algebra $A$. Let $u_t$ be implementing unitary
elements of $\alpha$, that is, $\alpha_t =\Ad(u_t )$ for any $t\in G$. Then
the crossed product of $A$ by $\alpha$, $A\rtimes_{\alpha}G$ is:
$$
A\rtimes_{\alpha}G=\{\sum_{t\in G}a_t u_t \, | \, a_t \in A \, \,\text{for any $t\in G$} \} .
$$
Let $A_t =Au_t$ for any $t\in G$. By routine computations, we see that
$\mathcal{A}_{\alpha}=\{A_t \}_{t\in G}$ is a $C^*$-algebraic bundle over $G$.
We call $\mathcal{A}_{\alpha}$ the $C^*$-algebraic bundle over $G$
\sl
induced
\rm
by an action $\alpha$.
Let $\beta$ be an action of $G$ on a unital $C^*$-algebra $B$ and let
$\mathcal{A}_{\beta}=\{B_t \}_{t\in G}$ induced by $\beta$, where $B_t =Bv_t$
for any $t\in G$ and $v_t$ are implementing unitary elements of $\beta$.
We suppose that $\alpha$ and $\beta$ are strongly Morita equivalent with respect to
an action $\lambda$ of $G$ on an $A-B$-equivalence bimodule $X$. Let $X\rtimes_{\lambda}G$ be the
crossed product of $X$ by $\lambda$ defined in Kajiwara and Watatani \cite [Definition 1.4]{KW2:discrete},
that is, the direct sum of $n$-copies of $X$ as a vector space, where $n$ is the order of $G$.
And its elements are written as formal sums so that
$$
X\rtimes_{\lambda}G =\{\sum_{t\in G}x_t w_t \, | \, x_t \in X \, \, \text{for any $t\in G$} \} ,
$$
where $w_t$ are indeterminates for all $t\in G$. Let $C=A\rtimes_{\alpha}G$,
$D=B\rtimes_{\beta}G$ and $Y=X\rtimes_{\lambda}G$. Then by \cite [Proposition 1.7]{KW2:discrete},
$Y$ is a $C-D$-equivalence bimodule, where we define the left $C$-action and the right $D$-action on $Y$
by
$$
(au_t )\cdot (xw_s )=(a\cdot \lambda_t (x))w_{ts} \quad
(xw_s)\cdot (bv_t )=(x\cdot \beta_s (b))v_{st}
$$
for any $a\in A$, $b\in B$, $x\in X$ and $t, s\in G$ and we define
the left $C$-valued inner product and the right $D$-valued inner product
on $Y$ by extending linearly the following:
$$
{}_C \la xw_t \, , \, yw_s \ra={}_A \la x \, , \, \lambda_{ts^{-1}}(y) \ra u_{ts^{-1}} , \quad
\la xw_t \, , \, yw_s \ra_D =\beta_{t^{-1}}(\la x, y \ra_B )v_{t^{-1}s}
$$
for any $x, y\in X$, $t, s\in G$. Let $X_t =Xw_t$ for any $t\in G$ and $\mathcal{X}_{\lambda}=\{X_t \}_{t\in G}$.
Then $Y=\oplus_{t\in G}X_t$.
Also, $\mathcal{X}_{\lambda}$ has Conditions (1R)-(3R) and (1L)-(3L) in \cite [Definition 2.1]{AF:bundle}.
Furthermore, $X$ is an $A-B$-equivalence bimodule, $\mathcal{X}_{\lambda}$ and satisfies
$$
{}_{C} \la X_t \, , \, X_s \ra=A_{ts^{-1}} \, , \quad \la X_t \, , \, X_s \ra_{D}=B_{t^{-1}s}
$$
for any $t, s\in G$. Therefore, $\mathcal{X}_{\lambda}$ is an $\mathcal{A}_{\alpha}-\mathcal{A}_{\beta}$-
equivalence bundle by Lemma \ref{lem:equivalence}.
\end{exam}

\section{Saturated $C^*$-algebraic bundles over a finite group}\label{sec:algebraic}
Let $\mathcal{A}=\{A_t \}_{t\in G}$ be a saturated $C^*$-algebraic
bundle over a finite group $G$. Let $e$ be the unit element in $G$.
Let $C=\oplus_{t\in G}A_t$ and $A=A_e$. We suppose hat
$C$ is unital and that $A$ has the unit element in $C$. Let $E^A$ be
the canonical conditional expectation from $C$ onto $A$  defined in Section
\ref{sec:intro}, which is of Watatani index-finite type. 
Let $C_1$ be the $C^*$-basic construction of $C$ and $e_A$
the Jones' projection for $E^A$. By \cite [Lemma 3.7]{KT7:characterization},
there is an action $\alpha^{\mathcal{A}}$ of $G$ on
$C_1$ induced by $\mathcal{A}$
defined as follows: Since $\mathcal{A}$ is saturated and
$A$ is unital, there is a finite set $\{x_i^t \}_{i=1}^{n_t}\subset A_t$
such that $\sum_{i=1}^{n_t}x_i^t x_i^{t*} =1$ for any $t\in G$.
Let $e_t =\sum_{i=1}^{n_t}x_i^t e_A x_i^{t*}$ for all $t\in G$.
Then by \cite [Lemmas 3.3, 3.5 and Remark 3.4]{KT7:characterization}, $\{e_t \}_{t\in G}$ are mutually
orthogonal projections in $A' \cap C_1$, which are independent of the choice of $\{x_i^t \}_{i=1}^{n_t}$,
with $\sum_{t\in G}e_t =1$ such that
$C$ and $e_t$ generate the $C^*$-algebra $C_1$ for all $t\in G$.
We define $\alpha^{\mathcal{A}}$ by
$\alpha_t^{\mathcal{A}}(c)=c$ and $\alpha_t^{\mathcal{A}}(e_A )=e_{t^{-1}}$
for any $t\in G$, $c\in C$. Let $\mathcal{A}_1 =\{Y_{\alpha_t^{\mathcal{A}}}\}_{t\in G}$
be the $C^*$-algebraic bundle over $G$
induced by the action $\alpha^{\mathcal{A}}$ of $G$ which is defined in
\cite [Sections 5, 6]{KT7:characterization}, that is, let $Y_{\alpha_t^{\mathcal{A}}}
=e_A C_1 \alpha_t^{\mathcal{A}}(e_A )=e_A C_1 e_{t^{-1}}$ for any $t\in G$.
The product $\bullet$ and the involution $\sharp$ in $\mathcal{A}_1$ are defined as follows:
\begin{align*}
(x, y)\in Y_{\alpha_t^{\mathcal{A}}}\times Y_{\alpha_s^{\mathcal{A}}} & \mapsto
x\bullet y=x\alpha_t^{\mathcal{A}}(y)\in Y_{\alpha_{ts}^{\mathcal{A}}} , \\
x\in Y_{\alpha_t^{\mathcal{A}}} & \mapsto x^{\sharp}=\alpha_{t^{-1}}^{\mathcal{A}}(x^* )\in
Y_{\alpha_{t^{-1}}^{\mathcal{A}}}
\end{align*}

\begin{lemma}\label{lem:iso}With the above notation, $\mathcal{A}$ and $\mathcal{A}_1$ are
isomorphic as $C^*$-algebraic bundles over $G$.
\end{lemma}
\begin{proof} Since $C_1 =Ce_A C$, for any $t\in G$
$$
Y_{\alpha_t^{\mathcal{A}}}=e_A Ce_A Ce_{t^{-1}}=e_A A Ce_{t^{-1}}=e_A Ce_{t^{-1}} .
$$
Let $x$ be any element in $C$. Then we can write that $x=\sum_{s\in G}x_s$,
where $x_s \in A_s$. Hence
\begin{align*}
e_A x e_{t^{-1}} & =\sum_{s, i}e_A x_s x_i^{t^{-1}}e_A x_i^{t^{-1}*} =
\sum_{s, i}E^A (x_s x_i^{t^{-1}})e_A x_i^{t^{-1}*} \\
& =\sum_i x_t x_i^{t^{-1}}e_A x_i^{t^{-1}*}=e_A x_t \sum_i x_i^{t^{-1}}x_i^{t^{-1}*}=e_A x_t .
\end{align*}
Thus $Y_{\alpha_t^{\mathcal{A}}}=e_A Ce_{t^{-1}}=e_A  A_t$ for any $t\in G$.
Let $\pi_t$ be the map from $A_t$ to $Y_{\alpha_t^{\mathcal{A}}}$ defined by
$$
\pi_t (x)=e_A x
$$
for any $x\in A_t$ and $t\in G$. By the above discussions $\pi_t$ is a linear map from $A_t$ onto
$Y_{\alpha_t^{\mathcal{A}}}$. Then
$$
||\pi_t (x)||^2 =||e_A xx^* e_A ||=||E^A (xx^* )e_A ||=||E^A (xx^* )||=||xx^* ||=||x||^2 .
$$
Hence $\pi_t$ is injective for any $t\in G$. Thus $A_t \cong e_A C_1 \alpha_t^{\mathcal{A}}(e_A )$
as Banach spaces for any $t\in G$. Also, for any $x\in A_t$, $y\in A_s$, $t, s\in G$,
\begin{align*}
\pi_t (x)\bullet \pi_s (y) & =e_A x\alpha_t^{\mathcal{A}}(e_A y)=e_A xe_{t^{-1}}y
=e_A \sum_i xx_i^{t^{-1}}e_A x_i^{t^{-1}*} y \\
& =e_A \sum_i xx_i^{t^{-1}}x_i^{t^{-1}*}y=e_A xy=\pi_{ts}(xy) ,
\end{align*}
\begin{align*}
\pi_t (x)^{\sharp} & =\alpha_{t^{-1}}^{\mathcal{A}}(\pi_t (x)^* )=\alpha_{t^{-1}}^{\mathcal{A}}((e_A x)^* )
=\alpha_{t^{-1}}^{\mathcal{A}}(x^* e_A )=x^* e_t \\
& =\sum_i x^* x_i^t e_A x_i^{t*}=e_A \sum_i x^* x_i^t x_i^{t*} =e_A x^* =\pi_{t^{-1}}(x^* ) .
\end{align*}
Therefore, $\mathcal{A}=\{A_t \}_{t\in G}$ and $\mathcal{A}_1 =\{Y_{\alpha_t^{\mathcal{A}}}\}_{t\in G}$ are
isomorphic  as $C^*$-algebraic bundles over $G$.
\end{proof}

\section{Strong Morita equivalence for unital inclusions of unital $C^*$-algberas}
\label{sec:inclusion}
Let $\mathcal{A}=\{A_t \}_{t\in G}$ and $\mathcal{B}=\{B_t \}_{t\in G}$ be saturated $C^*$-algebraic
bundles over a finite group $G$. Let $e$ be the unit element in $G$.
Let $C=\oplus_{t\in G}A_t$, $D=\oplus_{t\in G}B_t$ and $A=A_e$, $B=B_e$. We suppose that
$C$ and $D$ are unital and that $A$ and $B$ have the unit elements in
$C$ and $D$, respectively. Let $E^A$ and $E^B$ be
the canonical conditional expectations from $C$ and $D$ onto $A$ and $B$ defined in Section
\ref{sec:intro}, respectively. They are of Watatani index-finite type. Let $A\subset C$ and $B\subset D$ be the 
unital inclusions of unital $C^*$-algebras induced by $\mathcal{A}$ and $\mathcal{B}$, respectively.
We suppose that $A\subset C$ and $B\subset D$ are strongly Morita equivalent
with respect to a $C-D$-equivalence bimodule $Y$ and its closed subspace $X$.
Also, we suppose that $A' \cap C=\BC 1$. Then by \cite [Lemma 10.3]{KT4:morita},
$B' \cap D=\BC 1$ and by \cite [Lemma 4.1]{Kodaka:Picard2} and its proof, there is the unique
conditional expectation $E^X$ from $Y$ onto $X$ with respect to $E^A$ and $E^B$.
Let $C_1$ and $D_1$ be the $C^*$-basic constructions of $C$ and $D$ and $e_A$ and $e_B$
the Jones' projections for $E^A$ and $E^B$, respectively. Then by \cite [Lemma 3.7]{KT7:characterization},
there are actions $\alpha^{\mathcal{A}}$ and $\alpha^{\mathcal{B}}$ of $G$ on
$C_1$ and $D_1$ induced by $\mathcal{A}$ and $\mathcal{B}$, respectively.
Furthermore, let $C_2$ and $D_2$ be the $C^*$-basic constructions of $C_1$ amd $D_1$ for the
dual conditional expectations $E^C$ of $E^A$ and $E^D$ of $E^B$, which are isomorphic to
$C_1 \rtimes_{\alpha^{\mathcal{A}}}G$ and $D_1\rtimes_{\alpha^{\mathcal{B}}}G$, respectively.
We identify $C_2$ and $D_2$ with $C_1 \rtimes_{\alpha^{\mathcal{A}}}G$
and $D_1\rtimes_{\alpha^{\mathcal{B}}}G$, respectively. By \cite [Corollary 6.3]{KT4:morita},
the unital inclusions $C_1 \subset C_2$ and $D_1 \subset D_2$ are strongly Morita equivalent
with respect to a $C_2 -D_2$-equivalence bimodule $Y_2$ and its closed subspace $Y_1$,
where $Y_1$ and $Y_2$ are the $C_1 -D_1$-equivalence bimodule and
the $C_2 -D_2$-equivalence bimodule defined in \cite [Section 6]{KT4:morita}, respectively and
$Y_1$ is regarded as a closed subspace of $Y_2$ in the same way as in \cite [Section 6]{KT4:morita}.
Also, ${C_1} ' \cap C_2 =\BC 1$ by the proof of Watatani \cite [Proposition 2.7.3]{Watatani:index}
since $A' \cap C=\BC 1$. Hence by \cite [Corollary 6.5]{KT5:Hopf},
there is an automorphism $f$ of $G$ such that $\alpha^{\mathcal{A}}$ is strongly Morita
equivalent to $\beta$, where $\beta$ is the action of $G$ on $D_1$
induced by $\alpha^{\mathcal{B}}$ and $f$, which is defined by
$\beta_t (d)=\alpha_{f(t)}^{\mathcal{B}}(d)$ for any $t\in G$ and $d\in D_1$.
Let $\lambda$ be an action of $G$ on a $C_1 -D_1$-equivalence bimodule $Z$ with 
respect to $(C_1 , D_ 1 , \alpha^{\mathcal{A}}, \beta)$.
\par
Let $\mathcal{A}_1 =\{Y_{\alpha_t^{\mathcal{A}}}\}_{t\in G}$ and
$\mathcal{B}_1 =\{Y_{\alpha_t^{\mathcal{B}}}\}_{t\in G}$ be the $C^*$-algebraic
bundles over $G$ induced by the actions $\alpha^{\mathcal{A}}$ and $\alpha^{\mathcal{B}}$,
which are defined in Section \ref{sec:algebraic}. Furthermore, let $\mathcal{B}^f =\{B_{f(t)}\}_{t\in G}$ be
the $C^*$-algebraic bundle over $G$ induced by $\mathcal{B}$ and $f$ and
let $\mathcal{B}_1^f =\{Y_{\beta_t }\}_{t\in G}$ be the $C^*$-algebraic bundle over $G$
induced by the action $\beta$, which is defined in Section \ref{sec:algebraic}.
We shall construct an $\mathcal{A}_1 -\mathcal{B}_1^f $-equivalence bundle
$\mathcal{Z}=\{Z_t \}_{t\in G}$ over $G$. Let $Z_t =e_A \cdot Z \cdot \beta_t (e_B )$ for any
$t\in G$ and let $W=\oplus_{t\in G}Z_t$. Also, by Lemma \ref {lem:iso} and its proof
$\oplus_{t\in G}Y_{\alpha_t^{\mathcal{A}}}\cong C$ and $\oplus_{t\in G}Y_{\beta_t}\cong D$
as $C^*$-algebras. We identify $\oplus_{t\in G}Y_{\alpha_t^{\mathcal{A}}}$ and $\oplus_{t\in G}Y_{\beta_t}$
with $C$ and $D$, respectively.
We define the left $C$-action $\diamond$ and the left $C$-valued inner product
${}_{C} \la -, - \ra$ on $W$
by
\begin{align*}
e_A x \alpha_t^{\mathcal{A}}(e_A )\diamond [e_A \cdot z\cdot \beta_s (e_B )] & \overset{\text{def}} =
e_A x \alpha_t^{\mathcal{A}}(e_A )\cdot \lambda_t (e_A \cdot z\cdot \beta_s (e_B )) \\
& =e_A x\alpha_t^{\mathcal{A}}(e_A )\cdot \lambda_t (z)\cdot \beta_{ts}(e_B ) \\
& =e_A \cdot [x\alpha_t^{\mathcal{A}}(e_A )\cdot \lambda_t (z)]\cdot \beta_{ts}(e_B ) ,
\end{align*}
\begin{align*}
{}_{C} \la e_A \cdot z\cdot \beta_t (e_B) \, ,\, e_A \cdot w\cdot \beta_s (e_B ) \ra 
& \overset{\text{def}}={}_{C_1} \la e_A \cdot z\cdot \beta_t (e_B ) \, , \, \lambda_{ts^{-1}}(e_A \cdot w\cdot \beta_s(e_B ))\ra \\
& =e_A \, {}_{C_1} \la z\cdot \beta_t (e_B ) \, , \, \lambda_{ts^{-1}}(w)\cdot\beta_t (e_B ) \ra
\alpha_{ts^{-1}}^{\mathcal{A}}(e_A ) ,
\end{align*}
where $e_A x \alpha_t^{\mathcal{A}}(e_A )\in e_A C_1 \alpha_t^{\mathcal{A}}(e_A )$, \,
$e_A \cdot z\cdot\beta_s (e_B ), \, e_A \cdot w\cdot\beta_s (e_B )\in Z_s$, \,
$e_A \cdot z\cdot\beta_t (e_B )\in Z_t$.
Also, we define the right $D$-action, which is also denoted by the same symbol $\diamond$ and the
$D$-valued inner product $\la -, - \ra_{D}$ on $W$ by
\begin{align*}
[e_A \cdot z\cdot \beta_t (e_B )]\diamond e_B x\beta_s (e_B ) & \overset{\text{def}} =
e_A \cdot z\cdot\beta_t (e_B )\beta_t (x)\beta_{ts}(e_B ) \\
& =e_A \cdot [z\cdot\beta_t (e_B )\beta_t (x)]\cdot\beta_{ts}(e_B ) ,
\end{align*}
\begin{align*}
\la e_A \cdot z\cdot\beta_t (e_B ) \, , \, e_A \cdot w\cdot \beta_s (e_B ) \ra_{D} & \overset{\text{def}} =
\beta_{t^{-1}}(\la e_A \cdot z\cdot \beta_t (e_B ) \, , \, e_A \cdot w\cdot \beta_s (e_B ) \ra_{D_1}) \\
& =e_B \beta_{t^{-1}}(\la e_A \cdot z \ , \, e_A \cdot w \ra_{D_1})\beta_{t^{-1}s}(e_B ) ,
\end{align*}
where $e_Bx\beta_s (e_B )\in e_B D_1 \beta_s (e_B )$, $e_A \cdot z\cdot \beta_t (e_B )\in Z_t$,
$e_A \cdot w\cdot \beta_s (e_B )\in Z_s$. By the above definitions, $\mathcal{Z}$ has Conditions
(1R)-(3R) and (1L)-(3L) in \cite [Definition 2.1]{AF:bundle}. We show that $\mathcal{Z}$
has Conditions (4R) and (4L) in \cite [Definition 2.1]{AF:bundle} and that $\mathcal{Z}$ is an
$\mathcal{A}_1 -\mathcal{B}_1^f $-bundle in the same way as in Example \ref {exam:action}.

\begin{lemma}\label{lem:condition4} With the above notation, $\mathcal{Z}$ has Conditions $(4R)$ and
$(4L)$ in \cite [Definition 2.1]{AF:bundle}.
\end{lemma}
\begin{proof}
Let $e_A \cdot z\cdot \beta_t (e_B )\in Z_t$, $e_A \cdot w\cdot \beta_s (e_B )\in Z_s$ and
$e_B x\beta_r (e_B )\in e_B D_1 \beta_r (e_B )$, where $t, s, r\in G$.
Then
\begin{align*}
& \la e_A \cdot z\cdot \beta_t (e_B ) \, , \, [e_A \cdot w\cdot\beta_s (e_B )]\diamond e_B x\beta_r (e_B ) \ra_D \\
& =\la e_A \cdot z\cdot \beta_t (e_B ) \, , \, e_A \cdot w \cdot \beta_s (e_B )\beta_s (x)\beta_{sr}(e_B ) \ra_{D} \\
& =\beta_{t^{-1}}(\la e_A \cdot z\cdot \beta_t (e_B ) \, , \, 
e_A \cdot w\cdot \beta_s (e_B )\beta_s (x)\beta_{sr}(e_B ) \ra_{D_1}) \\
& =\beta_{t^{-1}}(\la e_A \cdot z\cdot\beta_t (e_B ) \, , \, e_A \cdot w\cdot\beta_s (e_B ) \ra_{D_1}\,
\beta_s (e_B x \beta_r (e_B ))) \\
& =\beta_{t^{-1}}(\la e_A \cdot z\cdot \beta_t (e_B ) \, , \, e_A \cdot w\cdot \beta_s (e_B )\ra_{D_1})
\beta_{t^{-1}s}(e_B x \beta_r (e_B )) \\
& =e_B \beta_{t^{-1}}(\la e_A \cdot z \, , \, e_A \cdot w \ra_{D_1})\, \beta_{t^{-1}s}(e_B x\beta_r (e_B )) \\
& =e_B \beta_{t^{-1}}(\la e_A \cdot z \, , \, e_A \cdot w \ra_{D_1})\, \beta_{t^{-1}s}(e_B )\bullet e_B x\beta_r (e_B ) \\
& =\la e_A \cdot z\cdot \beta_t (e_B ) \, , \, e_A \cdot w\cdot \beta_s (e_B ) \ra_{D}\bullet e_B x\beta_r (e_B ) .
\end{align*}
Also,
\begin{align*}
& \la e_A \cdot z\cdot \beta_t (e_B ) \, , \, e_A \cdot w\cdot \beta_s (e_B ) \ra_{D}^{\sharp} \\
& =\beta_{t^{-1}}(\la e_A \cdot z\cdot \beta_t (e_B ) \, , \, e_A \cdot w\cdot \beta_s (e_B ) \ra_{D_1})^{\sharp} \\
& =(e_B \beta_{t^{-1}}(\la e_A \cdot z \, , \, e_A \cdot w \ra_{D_1})\beta_{t^{-1}s}(e_B ) )^{\sharp} \\
& =\beta_{s^{-1}t}(\beta_{t^{-1}s}(e_B )\beta_{t^{-1}}(\la e_A \cdot w \, , \, e_A \cdot z \ra_{D_1})e_B ) \\
& =e_B \beta_{s^{-1}}(\la e_A \cdot w \, , \, e_A \cdot z \ra_{D_1})\beta_{s^{-1}t}(e_B ) \\
& =\la e_A \cdot w\cdot \beta_s (e_B ) \, , \, e_A \cdot z\cdot \beta_t (e_B ) \ra_D .
\end{align*}
Hence $\mathcal{Z}$ has Condition (4R) in \cite [Definition 2.1]{AF:bundle}.
Next, let $e_A \cdot z\cdot \beta_t (e_B )\in Z_t$, $e_A \cdot w\cdot \beta_s (e_B )\in Z_s$ and
$e_A x\alpha_r^{\mathcal{A}}(e_A )\in e_A C_1 \alpha_r^{\mathcal{A}}(e_A )$, where $t, s, r\in G$. Then
\begin{align*}
& {}_C \la e_A x\alpha_r^{\mathcal{A}}(e_A )\diamond [e_A \cdot z\cdot \beta_t (e_B )] \, , \,
e_A \cdot w\cdot \beta_s (e_B ) \ra \\
& ={}_{C} \la e_A x\alpha_r^{\mathcal{A}}(e_A )\cdot \lambda_r (z)\cdot \beta_{rt}(e_B ) \, , \,
e_A \cdot w\cdot \beta_s (e_B ) \ra \\
& ={}_{C_1} \la e_A x \alpha_r^{\mathcal{A}}(e_A )\cdot \lambda_r (z)\cdot \beta_{rt}(e_B ) \, , \, 
\lambda_{rts^{-1}}(e_A \cdot w\cdot\beta_s (e_B ))\ra \\
& =e_A x\alpha_r^{\mathcal{A}}(e_A )\,{}_{C_1} \la \lambda_r (e_A \cdot z\cdot \beta_t (e_B )) \, , \,
\lambda_r (\lambda_{ts^{-1}}(e_A \cdot w\cdot \beta_s (e_B )))\ra \\
& =e_A x\alpha_r^{\mathcal{A}}(e_A )\alpha_r^{\mathcal{A}}({}_{C_1} \la e_A \cdot z\cdot \beta_t (e_B ) \, , \,
\lambda_{ts^{-1}}(e_A \cdot w\cdot \beta_s (e_B ))\ra ) \\
& =e_A x\alpha_r^{\mathcal{A}}(e_A )\bullet {}_{C_1} \la e_A \cdot z\cdot\beta_t (e_B ) \, , \,
\lambda_{ts^{-1}}(e_A \cdot w\cdot \beta_s (e_B ))\ra \\
& =e_A x\alpha_r^{\mathcal{A}}(e_A )\bullet {}_C \la e_A \cdot z\cdot \beta_t (e_B ) \, , \,
e_A \cdot w\cdot \beta_s (e_B ) \ra .
\end{align*}
Also,
\begin{align*}
& {}_C \la e_A \cdot z\cdot\beta_t (e_B ) \, , \, e_A \cdot w\cdot\beta_s (e_B ) \ra^{\sharp} \\
& =(e_A \, {}_{C_1} \la z\cdot \beta_t (e_B ) \, , \, \lambda_{ts^{-1}}(w)\cdot \beta_t (e_B ) \ra
\alpha_{ts^{-1}}^{\mathcal{A}}(e_A ))^{\sharp} \\
& =\alpha_{st^{-1}}^{\mathcal{A}}(\alpha_{ts^{-1}}^{\mathcal{A}}(e_A ){}_{C_1} \la \lambda_{ts^{-1}}(w)\cdot
\beta_t (e_B ) \, ,\, z\cdot\beta_t (e_B ) \ra e_A ) \\
& =e_A \, {}_{C_1} \la w\cdot \beta_s (e_B ) \, , \, \lambda_{st^{-1}}(z)\cdot \beta_s (e_B ) \ra
\alpha_{st^{-1}}^{\mathcal{A}}(e_A ) \\
& ={}_C \la e_A \cdot w\cdot\beta_s (e_B ) \, , \, e_A \cdot z\cdot \beta_t (e_B ) \ra .
\end{align*}
Hence $\mathcal{Z}$ has Conditoin (4L) in \cite [Definition 2.1]{AF:bundle}.
\end{proof}

By Lemma \ref {lem:condition4}, $W$ is a $C-D$-bimodule having
Properties (1)-(6) in \cite [Lemma 1.3]{KW2:discrete}. In order to prove that
$\mathcal{Z}$ has Conditions (5R), (6R) and (5L), (6L) in \cite [Definition 2.1]{AF:bundle}
using \cite [Lemma 1.3]{KW2:discrete}, we show that $W$ has Properties (7)-(10) in
\cite[Lemma 1.3]{KW2:discrete}.

\begin{lemma}\label{lem:property} With the above notation, $W$ has the following:
\vskip 0.1cm
\noindent
$(1)$ $(e_A x\alpha_t^{\mathcal{A}}(e_A )\diamond [e_A \cdot z\cdot \beta_s (e_B )])\diamond e_B y\beta_r (e_B )$
\newline
\quad \quad \quad \quad\quad\quad  
$=e_A x\alpha_t^{\mathcal{A}}(e_A )\diamond ([e_A \cdot z\cdot \beta_s (e_B )]\diamond e_B y\beta_r (e_B ))$,
\vskip 0.1cm
\noindent
$(2)$ $\la e_A x\alpha_t^{\mathcal{A}}(e_A )\diamond [e_A \cdot z\cdot\beta_s (e_B )] \, , \,
e_A \cdot w \cdot \beta_r (e_B) \ra_{D}$
\newline
\quad \quad \quad \quad\quad\quad  
$=\la e_A \cdot z\cdot\beta_s (e_B ) \, , \,
(e_A x\alpha_t^{\mathcal{A}}(e_A ))^{\sharp}\diamond [e_A \cdot w \cdot \beta_r (e_B)] \ra_{D}$,
\vskip 0.1cm
\noindent
$(3)$ ${}_C \la e_A \cdot z\cdot \beta_s (e_B ) \, ,
[e_A \cdot w\cdot \beta_r (e_B )]\diamond e_B y\beta_t (e_B ) \ra$
\newline
\quad \quad \quad \quad\quad\quad  
$={}_C \la [e_A \cdot z\cdot \beta_s (e_B )]\diamond (e_B y\beta_t (e_B ))^{\sharp} \, ,
e_A \cdot w\cdot \beta_r (e_B ) \ra$,
\end{lemma}
where $x\in C_1$, $y\in D_1$, $z, w\in Z$, $t, s, r\in G$.
\begin{proof} We show the lemma by routine computations.
Let $x\in C_1$, $y\in D_1$, $z, w\in Z$, $t, s, r\in G$.
\newline
We prove (1):
\begin{align*}
& (e_A x\alpha_t^{\mathcal{A}}(e_A)\diamond [e_A \cdot z\cdot \beta_s (e_B )])\diamond
e_B y\beta_r (e_B ) \\
&=[e_A x\alpha_t^{\mathcal{A}}(e_A )\cdot \lambda_t (z)\cdot \beta_{ts}(e_B )]\diamond e_B y\beta_r (e_B ) \\
& =e_A x\alpha_t^{\mathcal{A}}(e_A )\cdot \lambda_t (z)\cdot \beta_{ts}(e_B )\beta_{ts}(y)\beta_{tsr}(e_B ) \\
& =e_A x\alpha_t^{\mathcal{A}}(e_A )\diamond [e_A \cdot z\cdot \beta_s (e_B )\beta_s (y)\beta_{sr}(e_B )] \\
& =e_A x\alpha_t^{\mathcal{A}}(e_A )\diamond([e_A \cdot z\cdot \beta_s (e_B )]\diamond e_B y\beta_r (e_B )) .
\end{align*}
We prove (2): 
\begin{align*}
& \la  e_A x \alpha_t^{\mathcal{A}}(e_A )\diamond [e_A \cdot z\cdot \beta_s (e_B )]\, ,\,
e_A \cdot w\cdot \beta_r (e_B ) \ra_{D} \\
& =\la e_A x\alpha_t^{\mathcal{A}}(e_A )\cdot \lambda_t (z)\cdot \beta_{ts}(e_B ) \, , \,
e_A \cdot w\cdot\beta_r (e_B ) \ra_{D} \\
& =\la e_A \cdot [x\alpha_t^{\mathcal{A}}(e_A )\cdot \lambda_t (z)]\cdot \beta_{ts}(e_B ) \, , \,
e_A \cdot w\cdot\beta_r (e_B ) \ra_{D} \\
& =\beta_{s^{-1}t^{-1}}(\la e_A \cdot [x\alpha_t^{\mathcal{A}}(e_A )\cdot\lambda_t (z)]\cdot\beta_{ts}(e_B ) \, , \,
e_A \cdot w\cdot \beta_r (e_B ) \ra_{D_1}) \\
& =\beta_{s^{-1}t^{-1}}(\la \alpha_t^{\mathcal{A}}(e_A )\cdot \lambda_t (z)\cdot \beta_{ts}(e_B ) \, , \,
\alpha_t^{\mathcal{A}}(e_A )x^* e_A \cdot w\cdot\beta_r (e_B ) \ra_{D_1}) \\
& =\beta_{s^{-1}}(\la e_A \cdot z\cdot \beta_s (e_B ) \, , \, 
e_A \alpha_{t^{-1}}^{\mathcal{A}}(x^* )\alpha_{t^{-1}}^{\mathcal{A}}(e_A )\cdot \lambda_{t^{-1}}(w)
\cdot \beta_{t^{-1}r}(e_B ) \ra_{D_1}) \\
& =\beta_{s^{-1}}(\la e_A \cdot z\cdot \beta_s (e_B ) \, , \,
e_A \alpha_{t^{-1}}^{\mathcal{A}}(x^* e_A )\diamond [e_A \cdot w\cdot\beta_r (e_B )] \ra_{D_1}) \\
& =\beta_{s^{-1}}(\la e_A \cdot z\cdot\beta_s (e_B ) \, , \,
(e_A x\alpha_t^{\mathcal{A}}(e_A ))^{\sharp}\diamond [e_A \cdot w\cdot\beta_r (e_B )]\ra_{D_1}) \\
& =\la e_A \cdot z\cdot\beta_s (e_B ) \, , \,
(e_A x\alpha_t^{\mathcal{A}}(e_A ))^{\sharp}\diamond [e_A \cdot w\cdot\beta_r (e_B )] \ra_{D} .
\end{align*}
We prove (3):
\begin{align*}
& {}_{C} \la e_A \cdot z\cdot \beta_s (e_B ) \, , \, [e_A \cdot w\cdot\beta_r (e_B )]\diamond e_B y\beta_t (e_B ) \ra \\
& ={}_{C} \la e_A \cdot z\cdot\beta_s (e_B ) \, , \, e_A \cdot w\cdot\beta_r (e_B )\beta_r (y)\beta_{rt}(e_B ) \ra \\
& ={}_{C_1} \la e_A \cdot z\cdot\beta_s (e_B ) \, , \, \lambda_{st^{-1}r^{-1}}(e_A \cdot w\cdot\beta_r (e_B )\beta_r (y)
\beta_{rt}(e_B ) \ra \\
& ={}_{C_1} \la e_A \cdot z\cdot \beta_s (e_B ) \, , \, \alpha_{st^{-1}r^{-1}}(e_A )\cdot \lambda_{st^{-1}r^{-1}}(w)
\cdot\beta_{st^{-1}}(e_B )\beta_{st^{-1}}(y)\beta_s (e_B ) \ra \\
& ={}_{C_1} \la e_A \cdot z\cdot \beta_s (e_B )\beta_{st^{-1}}(y^* )\beta_{st^{-1}}(e_B ) \, , \,
\alpha_{st^{-1}r^{-1}}(e_A )\cdot\lambda_{st^{-1}r^{-1}}(w)\cdot\beta_{st^{-1}}(e_B ) \ra \\
& ={}_{C} \la e_A \cdot [z\cdot \beta_s (e_B )\beta_{st^{-1}}(y^* )]\cdot\beta_{st^{-1}}(e_B ) \, , \,
e_A \cdot w\cdot \beta_r (e_B ) \ra \\
& ={}_C \la [e_A \cdot z\cdot \beta_s (e_B )]\diamond (e_B y\beta_t (e_B ))^{\sharp} \, , \,
e_A \cdot w\cdot \beta_r (e_B ) \ra .
\end{align*}
Therefore, we obtain the conclusion.
\end{proof}

By Lemma \ref {lem:property}, $W$ has Properties (7), (8) in \cite [Lemma 1.3]{KW2:discrete}.

\begin{lemma}\label{lem:basis}With the above notation, there are finite subsets
$\{u_i \}_i$ and $\{v_j \}_j$ of $W$ such that
$$
\sum_i u_i \diamond \la u_i ,\,  x \ra_{D}=x 
=\sum_j {}_C \la x ,\, v_j \ra \diamond v_j 
$$
for any $x\in W$.
\end{lemma}
\begin{proof}
Since $Z$ is a $C_1 -D_1$-equivalence bimodule, there are finite subsets $\{z_i \}_i$ and
$\{w_j \}_j$ of $Z$ such that
$$
\sum_i z_i \cdot \la z_i , z \ra_{D_1} =z=\sum_j {}_{C_1} \la z, w_j \ra\cdot w_j
$$
for any $z\in Z$. Then for any $z\in Z$, $s\in G$,
\begin{align*}
& \sum_{i, t}[e_A \cdot z_i \cdot \beta_t (e_B )]\diamond \la e_A \cdot z_i \cdot\beta_t (e_B ) \, , \,
e_A \cdot z\cdot \beta_s (e_B ) \ra_{D} \\
& =\sum_{i, t}e_A \cdot z_i \cdot \beta_t (e_B )\diamond\beta_{t^{-1}}(\la e_A \cdot z_i \cdot\beta_t (e_B ) \, , \,
e_A \cdot z\cdot \beta_s (e_B ) \ra_{D_1} ) \\
& =\sum_{i, t}e_A \cdot z_i \cdot\beta_t (e_B )\diamond e_B \beta_{t^{-1}}(\la e_A \cdot z_i \, , \,
e_A \cdot z \ra_{D_1})\beta_{t^{-1}s}(e_B ) \\
& =\sum_{i, t}e_A \cdot z_i \cdot\beta_t (e_B )\la e_A \cdot z_i \, , \,
e_A \cdot z \ra_{D_1}\beta_s (e_B ) \\
& =\sum_{i, t}e_A \cdot [z_i \cdot \la z_i \cdot\beta_t (e_B ) \, , \, e_A \cdot z \ra_{D_1}]\cdot \beta_s (e_B ) \\
& =\sum_i e_A \cdot [z_i \cdot \la z_i \, , \, e_A \cdot z \ra_{D_1}]\cdot \beta_s (e_B ) \\
& =e_A \cdot z\cdot \beta_s (e_B )
\end{align*}
since $\sum_{t\in G}\beta_t (e_B )=1$ by \cite [Remark 3.4]{KT7:characterization}. Also, for any $z\in Z$, $s\in G$,
\begin{align*}
& \sum_{j, t}{}_C \la e_A \cdot z\cdot\beta_s (e_B ) \, , \, e_A \cdot\lambda_t (w_j )\cdot \beta_t (e_B ) \ra \diamond
[e_A  \cdot\lambda_t (w_j )\cdot \beta_t (e_B ) ] \\
& =\sum_{j, t}{}_{C_1} \la e_A \cdot z\cdot \beta_s (e_B ) \, , \, \lambda_{st^{-1}}(e_A \cdot \lambda_t (w_j )\cdot
\beta_t (e_B ))\ra \diamond [e_A \cdot \lambda_t (w_j )\cdot \beta_t (e_B )] \\
& =\sum_{j, t}{}_{C_1} \la e_A \cdot z\cdot \beta_s (e_B ) \, , \, \alpha_{st^{-1}}^{\mathcal{A}}(e_A )\cdot
\lambda_s (w_j )\cdot \beta_s (e_B ) \ra\diamond [e_A \cdot \lambda_t (w_j )\cdot \beta_t (e_B )] \\
& =\sum_{j, t}e_A \, {}_{C_1} \la z\cdot \beta_s (e_B ) \, , \, \lambda_s (w_j ) \ra \alpha_{st^{-1}}^{\mathcal{A}}(e_A )
\diamond [e_A \cdot \lambda_t (w_j )\cdot \beta_t (e_B )] \\
& =\sum_{j, t}e_A \, {}_{C_1} \la z\cdot \beta_s (e_B ) \, , \, \lambda_s (w_j ) \ra\alpha_{st^{-1}}^{\mathcal{A}}(e_A )
\cdot \lambda_s (w_j )\cdot\beta_s (e_B ) \\
& =\sum_j e_A \, {}_{C_1} \la z\cdot \beta_s (e_B ) \, , \, \lambda_s (w_j ) \ra\cdot \lambda_s (w_j )\cdot
\beta_s (e_B ) \\
& =\sum_j e_A \cdot \lambda_s ({}_{C_1} \la \lambda_{s^{-1}}(z)\cdot e_B \, , \, w_j \ra\cdot w_j \cdot e_B ) \\
& =e_A \cdot \lambda_s (\lambda_{s^{-1}}(z)\cdot e_B )=e_A \cdot z\cdot \beta_s (e_B )
\end{align*}
since $\sum_{t\in G}\alpha_{st^{-1}}^{\mathcal{A}}(e_A )=1$ for any $s\in G$ by
\cite [Remark 3.4]{KT7:characterization}.
Therefore, we obtain the conclusion.
\end{proof}

\begin{remark}\label{rem:basis} By Lemma \ref{lem:property}, $\{e_A \cdot z_i \cdot\beta_t (e_B ) \}_{i, t}$
is a right $D$-basis and $\{e_A \cdot \lambda_t (w_j )\cdot \beta_t (e_B )\}_{j, t}$ is a left $C$-basis of $W$
in the sense of Kajiwara and Watatani \cite {KW1:bimodule}.
\end{remark}

By Lemma \ref{lem:property}, $W$ has Properties (9), (10) in \cite [Lemma 1.3]{KW2:discrete}.
Hence by \cite [Lemma 1.3]{KW2:discrete}, $W$ is a Hilbert $C -D$-
bimodule in the sense of \cite [Definition 1.1]{KW2:discrete}.
Thus, $\mathcal{Z}$ has Conditions (5R), (6R) and (5L), (6L) in \cite [Definition 2.1]{AF:bundle}.

\begin{lemma}\label{lem:full} With the above notation, $\mathcal{Z}$ is an $\mathcal{A}_1 - \mathcal{B}_1^f$-equivalence bundle such that
$$
{}_{\mathcal{A}_1} \la Z_t , \, Z_s \ra =Y_{\alpha_{ts^{-1}}^{\mathcal{A}}} , \quad
\la Z_t ,\, Z_s \ra_{\mathcal{B}_1^f} =Y_{\beta_{t^{-1}s}}
$$
for any $t, s\in G$.
\end{lemma}
\begin{proof}First, we show that the left $C$-valued inner product and the right $D$-valued inner product on $W$
are compatible. Let $y, z, w\in Z$ and $t, s, r \in G$. Since $Z$ is a $C_1 -D_1$-equivalence bimodule,
\begin{align*}
& {}_C \la e_A \cdot z\cdot\beta_t (e_B ) \, , \, e_A \cdot y\cdot\beta_s (e_B ) \ra\diamond
[e_A \cdot w\cdot \beta_r (e_B )] \\
& ={}_{C_1} \la e_A \cdot z\cdot \beta_t (e_B ) \, , \, \lambda_{ts^{-1}}(e_A  \cdot y\cdot \beta_s (e_B )) \ra
\diamond [e_A \cdot w\cdot \beta_r (e_B )] \\
& ={}_{C_1} \la e_A \cdot z\cdot \beta_t (e_B ) \, , \, \alpha_{ts^{-1}}^{\mathcal{A}}(e_A )\cdot
\lambda_{ts^{-1}}(y)\cdot \beta_t (e_B ) \ra \diamond [e_A \cdot w\cdot \beta_r (e_B )] \\
& =e_A \, {}_{C_1} \la z\cdot \beta_t (e_B ) \, , \, \lambda_{ts^{-1}}(y)\cdot \beta_t (e_B ) \ra
\alpha_{ts^{-1}}^{\mathcal{A}}(e_A )\diamond [e_A \cdot w\cdot \beta_r (e_B )] \\
& =e_A \, {}_{C_1} \la z\cdot \beta_t (e_B ) \, , \, \lambda_{ts^{-1}}(y)\cdot \beta_t (e_B ) \ra
\alpha_{ts^{-1}}^{\mathcal{A}}(e_A )\cdot\lambda_{ts^{-1}}(w)\cdot \beta_{ts^{-1}r}(e_B ) \\
& ={}_{C_1} \la e_A \cdot z\cdot \beta_t (e_B ) \, , \, \alpha_{ts^{-1}}^{\mathcal{A}}(e_A )\cdot
\lambda_{ts^{-1}}(y)\cdot\beta_t (e_B ) \ra\cdot [\alpha_{ts^{-1}}^{\mathcal{A}}(e_A )\cdot
\lambda_{ts^{-1}}(w)\cdot\beta_{ts^{-1}r}(e_B )] \\
& =[e_A \cdot z\cdot \beta_t (e_B )]\cdot \la \alpha_{ts^{-1}}^{\mathcal{A}}(e_A ) \cdot \lambda_{ts^{-1}}(y)
\cdot \beta_t (e_B) \, , \, \alpha_{ts^{-1}}^{\mathcal{A}}(e_A )\cdot \lambda_{ts^{-1}}(w)\cdot
\beta_{ts^{-1}r}(e_B ) \ra_{D_1} \\
& =[e_A \cdot z\cdot \beta_t (e_B )]\cdot\beta_{ts^{-1}}(\la e_A \cdot y\cdot\beta_s (e_B ) \, , 
e_A \cdot w\cdot \beta_r (e_B) \ra_{D_1} \\
& =[e_A \cdot z\cdot \beta_t (e_B )]\diamond \la e_A \cdot y\cdot\beta_s (e_B ) \, , \, e_A \cdot w
\cdot\beta_r (e_B ) \ra_{D} .
\end{align*}
Hence the left $C$-valued inner product and the right $D$-valued inner product are compatible.
Next, we show that
$$
{}_{\mathcal{A}_1} \la Z_t , \, Z_s \ra =Y_{\alpha_{ts^{-1}}^{\mathcal{A}}} , \quad
\la Z_t , \, Z_s \ra_{\mathcal{B}_1^f} =Y_{\beta_{t^{-1}s}}
$$
for any $t, s\in G$. Let $t, s \in G$. Since $E^B$ is of Watatani index-finite type,
there is a quasi-basis $\{(d_j , d_j^* )\}\subset D\times D$ for $E^B$. Thus
$\sum_j d_j e_B d_j^* =1$. Since $Z$ is a $C_1 -D_1$-equivalence bimodule,
there is a finite subset $\{z_i \}$ of $Z$ such that
$\sum_i {}_{C_1} \la z_i , z_i \ra =1$. Let $c\in C$. Then
\begin{align*}
& \sum_{i, j}{}_C \la e_A c\cdot \lambda_t (z_i )\cdot \beta_t (d_j e_B ) \, , \, e_A \cdot
\lambda_s (z_i )\cdot \beta_s (d_j e_B ) \ra \\
& =\sum_{i, j}{}_{C_1} \la e_A c\cdot \lambda_t (z_i )\cdot \beta_t (d_j e_B ) \, , \, \lambda_{ts^{-1}}
(e_A \cdot \lambda_s (z_i )\cdot \beta_s (d_j e_B )) \ra \\
& =\sum_{i, j}{}_{C_1}\la e_A c \cdot \lambda_t (z_i )\cdot\beta_t (d_j e_B ) \, , \,
\alpha_{ts^{-1}}^{\mathcal{A}}(e_A )\cdot \lambda_t (z_i )\cdot\beta_t (d_j e_B ) \ra \\
& =\sum_{i, j}e_A \, {}_{C_1} \la c\cdot \lambda_t (z_i )\cdot\beta_t (d_j e_B d_j^* ) \, , \,
\lambda_t (z_i )\ra \alpha_{ts^{-1}}^{\mathcal{A}}(e_A ) \\
& =\sum_i e_A c \, {}_{C_1} \la \lambda_t (z_i ) \, , \, \lambda_t (z_i ) \ra\alpha_{ts^{-1}}^{\mathcal{A}}(e_A ) \\
& =\sum_i e_A c\alpha_t^{\mathcal{A}}({}_{C_1} \la z_i \, , \, z_i \ra )\alpha_{ts^{-1}}^{\mathcal{A}}(e_A ) \\
& =e_A c \alpha_{ts^{-1}}^{\mathcal{A}}(e_A ) .
\end{align*}
Hence we obtain that ${}_C \la Z_t , \, Z_s \ra =Y_{\alpha_{ts^{-1}}^{\mathcal{A}}}$ for any $t, s\in G$.
Also, since $E^A$ is of Watatani index-finite type, there is a quasi-basis $\{ (c_j , c_j^* ) \}\subset C\times C$ for $E^A$.
Thus $\sum_j c_j e_A c_j^* =1$. Since $Z$ is a $C_1 -D_1$-equivalence bimodule,
there is a finite subset $\{w_i \}$ of $Z$ such that
$\sum_i \la w_i , w_i \ra_{D_1}  =1$. Let $d\in D_1$. Then
\begin{align*}
& \sum_{i, j}\la e_A c_j^* \cdot w_i \cdot\beta_t (e_B ) \, , \, e_A c_j^* \cdot w_i \cdot d\beta_s (e_B ) \ra_{D} \\
& =\sum_{i, j}\beta_{t^{-1}}(\la e_A c_j^* \cdot w_i \cdot \beta_t (e_B ) \, , \,
e_A c_j^* \cdot w_i \cdot d\beta_s (e_B ) \ra_{D_1} ) \\
& =\sum_{i, j}e_B \beta_{t^{-1}}(\la e_A c_j^* \cdot w_i \, , \,
e_A c_j^* \cdot w_i \ra_{D_1})\beta_{t^{-1}}(d)\beta_{t^{-1}s}(e_B ) \\
& =\sum_{i, j}e_B \beta_{t^{-1}}(\la w_i \, , \, c_j e_A c_j^* \cdot w_i \ra_{D_1})\beta_{t^{-1}}(d)\beta_{t^{-1}s}(e_B ) \\
& =\sum_i e_B \beta_{t^{-1}}(\la w_i , w_i \ra_{D_1})\beta_{t^{-1}}(d)\beta_{t^{-1}s}(e_B ) \\
& =e_B \beta_{t^{-1}}(d)\beta_{t^{-1}s}(e_B ) .
\end{align*}
Hence we obtain that $\la Z_t , \,  Z_s \ra_{D}=Y_{\beta_{t^{-1}s}}$ for any $t, s\in G$.
Therefore, we obtain the conclusion.
\end{proof}
Combining the above lemmas, we obtain the following:

\begin{prop}\label{prop:combine} With the above notation, $\mathcal{Z}$ is an
$\mathcal{A}_1 -\mathcal{B}_1^f$-equivalence bundle over $G$.
\end{prop}
\begin{proof}
This is immediate by Lemmas \ref{lem:condition4}, \ref {lem:property}, \ref {lem:basis}, \ref {lem:full}.
\end{proof}

\begin{thm}\label{thm:inverse} Let $\mathcal{A}=\{A_t \}_{t\in G}$ and $\mathcal{B}=\{B_t \}_{t\in G}$
be saturated $C^*$-algebraic bundles over a finite group $G$. Let $e$ be the unit element in $G$.
Let $C=\oplus_{t\in G}A_t$, $D=\oplus_{t\in G}B_t$
and $A=A_e$, $B=B_e$. We suppose that $C$ and $D$ are unital and that $A$ and $B$ have the
unit elements in $C$ and $D$, respectively. Let $A\subset C$ and $B\subset D$ be the unital inclusions
of unital $C^*$-algebras induced by $\mathcal{A}$ and $\mathcal{B}$, respectively. Also, we suppose that
$A' \cap C=\BC 1$. If $A\subset C$ and $B\subset D$ are strongly Morita equivalent, then
there are an automorphism $f$ of $G$ and an $\mathcal{A}-\mathcal{B}^f$-equivalence bundle
$\mathcal{Z}=\{Z_t \}_{t\in G}$ satisfying that
$$
{}_C \la Z_t , \, Z_s \ra=A_{ts^{-1}} , \quad \la Z_t ,\,  Z_s \ra_{D}=B_{f(t^{-1}s)}
$$
for any $t, s\in G$, where $\mathcal{B}^f$ is the $C^*$-algebraic bundle over $G$ induced by
$\mathcal{B}$ and $f$ defined by $\mathcal{B}^f =\{B_{f(t)}\}_{t\in G}$.
\end{thm}
\begin{proof}
This is immediate by Lemma \ref{lem:iso} and Proposition \ref{prop:combine}.
\end{proof}

\section{Application}\label{sec:appli} Let $A$ and $B$ be unital $C^*$-algebras
and $X$ a Hilbert $A-B$-bimodule. Let $\widetilde{X}$ be its dual Hilbert
$B-A$-bimodule. For any $x\in X$, $\widetilde{x}$ denotes the element
in $\widetilde{X}$ induced by $x\in X$. 

\begin{lemma}\label{lem:tilde} Let $A$, $B$ and $C$ be unital $C^*$-algebras.
Let $X$ be a Hilbert $A-B$-bimodule and $Y$ a Hilbert $B-C$-bimodule. Then
$\widetilde{X\otimes_B Y}\cong \widetilde{Y}\otimes_B \widetilde{X}$ as Hilbert
$C-A$-bimodules.
\end{lemma}
\begin{proof}Let $\pi$ be the map from $\widetilde{X\otimes_B Y}$ to
$\widetilde{Y}\otimes_B \widetilde{X}$ defined by
$$
\pi(\widetilde{x\otimes y})=\widetilde{y}\otimes\widetilde{x}
$$
for any $x\in X$, $y\in Y$. Then by routine computaions, we can see that
$\pi$ is a Hilbert $C-A$-bimodule isomorphism of $\widetilde{X\otimes_B Y}$ onto
$\widetilde{Y}\otimes_B \widetilde{X}$.
\end{proof}

We identify $\widetilde{X\otimes_B Y}$ with $\widetilde{Y}\otimes_B \widetilde{X}$
by the isomorphism $\pi$ defined in the proof of Lemma \ref {lem:tilde}.
Next, we give the definition of an involutive Hilbert $A-A$-bimodule
modifying \cite {KT6:involutive}.

\begin{Def}\label{def:involutive} We say that a Hilbert $A-A$-bimodule $X$ is
\sl
involutive
\rm
if there exists a conjugate linear map $x\in X \mapsto x^{\natural}\in X$ such that
\newline
(1) $(x^{\natural})^{\natural}=x$, \, $x\in X$,
\newline
(2) $(a\cdot x\cdot b)^{\natural}=b^* \cdot x^{\natural}\cdot a^*$, \, $x\in X$, $a, b\in A$,
\newline
(3) ${}_A \la x, \, y^{\natural} \ra =\la x^{\natural} , \, y \ra_A$, \, $x, y\in X$.
\end{Def}

We call the above conjugate linear map $\natural$ an
\sl
involution
\rm
on $X$. If $X$ is full with the both inner products, $X$ is an involutive $A-A$-equivalence
bimodule. For each involutive Hilbert $A-A$-bimodule, let $L_X$ be the linking $C^*$-algebra
induced by $X$ and $C_X$ the $C^*$-subalgebra of $L_X$, which is defined in \cite {KT6:involutive},
that is,
$$
C_X =\{\begin{bmatrix} a & x \\
\widetilde{x}^{\natural} & a 
\end{bmatrix}\, | \, a\in A \, , \, x\in X\} ,
$$
We note that $C_X$ acts on $X\oplus A$ (See Brown, Green and Rieffel \cite {BGR:linking}
and Rieffel \cite {Rieffel:rotation}). The norm of $C_X$ is defined as the operator norm on $X\oplus A$.

Let $A$ be a unital $C^*$-algebra and $X$ an involutive Hilbert $A-A$-bimodule.
Let $\widetilde{X}$ be its dual Hilbert $A-A$-bimodule. We define the map ${}^{\natural}$ on
$\widetilde{X}$ by $(\widetilde{x})^{\natural}=\widetilde{(x^{\natural})}$ for any $\widetilde{x}\in\widetilde{X}$.

\begin{lemma}\label{lem:involution2}With the above notation, the above map ${}^{\natural}$
is an involution on $\widetilde{X}$.
\end{lemma}
\begin{proof}This is immediate by direct computations.
\end{proof}

For each involutive Hilbert $A-A$-bimodule $X$, we regard $\widetilde{X}$ as an involutive $A-A$-bimodule
in the same manner of Lemma \ref{lem:involution2}.

Let $\BZ_2=\BZ/2\BZ$ and $\BZ_2$ consists of the unit element $0$ and $1$.
Let $X$ be an involutive Hilbert $A-A$-bimodule. We construct a $C^*$-algebraic bundle over $\BZ_2$
induced by $X$. Let $A_0 =A$ and $A_1 =X$. Let $\mathcal{A}_X =\{A_t \}_{t\in \BZ_2}$.
We define a product $\bullet$ and an involution $\sharp$ as follows:
\newline
(1) $a\bullet b=ab$ \quad $a, b\in A$,
\newline
(2) $a\bullet x =a\cdot x$, \quad $x\bullet a =x\cdot a$ \quad $a\in A$, $x\in X$,
\newline
(3) $x\bullet y={}_A \la x, y^{\natural} \ra =\la x^{\natural}, y \ra_A$ \quad $x, y\in X$,
\newline
(4) $a^{\sharp}=a^*$ \quad $a\in A$,
\newline
(5) $x^{\sharp}=x^{\natural}$ \quad $x\in X$.
\newline
Then $A\oplus X$ is a $*$-algebra and by routine computations,
$A\oplus X$ is isomorphic to $C_X$ as $*$-algebras.
We identify $A\oplus X$ with $C_X$ as $*$-algebras. We define a norm
of $A\oplus X$ as the operator norm on $X\oplus A$. Hence
$\mathcal{A}_X$ is a $C^*$-algebraic bundle over $\BZ_2$.
Thus, we obtain a correspondence from the involutive Hilbert $A-A$-bimodules to
the $C^*$-algebraic bundles over $\BZ_2$. Next, let $\mathcal{A}=\{A_t \}_{t\in \BZ_2}$
be a $C^*$-algebraic bundle over $\BZ_2$. Then $A_1$ ia an involutive Hilbert $A-A$-bimodule.
Hence we obtain a correspondence from the $C^*$-algebraic bundles over $\BZ_2$ to
the involutive Hilbert $A-A$-bimodules. Clearly the above two correspondences are the
inverse correspondences of each other. Furthermore, the inclusion of unital $C^*$-algebras
$A\subset C_X$ induced by $X$ and the inclusion of unital $C^*$-algebras $A\subset A\oplus X$ induced by
the $C^*$-algebraic bundle $\mathcal{A}_X$ coincide.

\begin{lemma}\label{lem:iso2} Let $X$ and $Y$ be involutive Hilbert $A-A$-bimodules and
$\mathcal{A}_X$ and $\mathcal{A}_Y$ the $C^*$-algebraic bundles over $\BZ_2$ induced by
$X$ and $Y$, respectively. Then $\mathcal{A}_X \cong \mathcal{A}_Y$ as $C^*$-algebraic bundles over $\BZ_2$
if and only if $X\cong Y$ as involutive Hilbert $A-A$-bimodules.
\end{lemma}
\begin{proof} We suppose that $\mathcal{A}_X \cong \mathcal{A}_Y$ as $C^*$-algebraic
bundles over $\BZ_2$. Then there is a $C^*$-algebraic bundle isomorphism $\{\pi_t\}_{t\in \BZ_2}$
of $\mathcal{A}_X$ onto $\mathcal{A}_Y$. We identify $A$ with $\pi_0 (A)$. Then $\pi_1 $
is an involutive Hilbert $A-A$-bimodule isomorphism of $X$ onto $Y$.
Next, we suppose that there is an involutive Hilbert $A-A$-bimodule isomorphism $\pi$ of
$X$ onto $Y$. Let $\pi_0 =\id_A$ and $\pi_1 =\pi$. Then $\{\pi_t \}_{t\in \BZ_2}$ is
a $C^*$-algebraic bundle isomorphism $\mathcal{A}_X$ onto $\mathcal{A}_Y$.
\end{proof}

\begin{lemma}\label{lem:saturated} Let $X$ be an involutive Hilbert $A-A$-bimodule
and $\mathcal{A}_X$ the $C^*$-algebraic bundle over $\BZ_2$ induced by
$X$. Then $X$ is full with the both inner products if and only if $\mathcal{A}_X$ is saturated.
\end{lemma}
\begin{proof}
We suppose that $X$ is full with the both inner products. Then
$$
A_1 \bullet A_1^{\sharp}={}_A \la X \, , \, X \ra=A=A_0 .
$$
Also,
\begin{align*}
A_0 \bullet A_1^{\sharp} & =A\cdot X^{\natural}=A\cdot X=X=A_1 , \\
A_1 \bullet A_0^{\sharp} & =X\cdot A^* =X\cdot A=X=A_1
\end{align*}
by \cite [Proposition1.7]{BMS:quasi}. Clearly $A_0 \bullet A_0 =AA=A=A_0$.
Hence $\mathcal{A}_X$ is saturated. Next, we suppose that $\mathcal{A}_X$ is saturated.
Then
\begin{align*}
{}_A \la X \, , \, X \ra & =A_1 \bullet A_1^{\sharp} =A_1 =A , \\
\la X \, , \, X \ra_A  & ={}_A \la X^{\natural} \, , \, X^{\natural} \ra={}_A \la X \, ,\ \, X \ra=A .
\end{align*}
Thus $X$ is full with the both inner products.
\end{proof}

\begin{remark}\label{remark:inner} Let $X$ be an involutive Hilbert $A-A$-bimodule.
Then by the above proof, we see that $X$ is full with the left $A$-valued inner product if and only if
$X$ is full with the right $A$-valued inner product.
\end{remark}

\begin{lemma}\label{lem:involution} Let $A$ and $B$ be unital $C^*$-algebras and $M$ an
$A-B$-equivalence bimodule. Let $X$ be an involutive Hilbert $A-A$-bimodule. Then $\widetilde{M}\otimes_A X\otimes_A M$ is an involutive Hilbert $B-B$-bimodule whose involution $\natural$ is defined by
$$
(\widetilde{m}\otimes x  \otimes n)^{\natural}=\widetilde{n}\otimes x^{\natural}\otimes m
$$
for any $m, n\in M$, $x\in X$.
\end{lemma}
\begin{proof}
This is immediate by routine computations.
\end{proof}

Let $A, B, X$ and $M$ be as in Lemma \ref {lem:involution}. Let $Y$ be an involutive Hilbert $B-B$-bimodule.
We suppose that there is an involutive Hilbert $B-B$-bimodule isomorphism $\Phi$ of
$\widetilde{M}\otimes_A X\otimes_A M$ onto $Y$. 
Let $\widetilde{\Phi}$ be the linear map from $\widetilde{M}\otimes_A \widetilde{X}\otimes_A M$ onto $\widetilde{Y}$
defined by
$$
\widetilde{\Phi}(\widetilde{m}\otimes \widetilde{x}\otimes n)
=\widetilde{\Phi}((\widetilde{n}\otimes x\otimes m )^{\widetilde{}})
=[\Phi(\widetilde{n}\otimes x\otimes m)]^{\widetilde{}}
$$
for any $m, n \in M$, $x\in X$.

\begin{lemma}\label{lem:conjugate} With the above notation, $\widetilde{\Phi}$ is an
involutive Hilbert $B-B$-bimodule isomorphism of $\widetilde{M}\otimes_A \widetilde{X}\otimes_A M$ onto
$\widetilde{Y}$.
\end{lemma}
\begin{proof} This is immediate by routine computations
\end{proof}

Again, let $A, B, X$ and $M$ be as in Lemma \ref {lem:involution}. Let $Y$ be an involutive Hilbert $B-B$-bimodule.
We suppose that there is an involutive Hilbert $B-B$-bimodule isomorphism $\Phi$ of
$\widetilde{M}\otimes_A X\otimes_A M$ onto $Y$. 
Then there is a finite subset $\{u_i \}$ of $M$ with
$\sum_i {}_A \la u_i \, , \, u_i \ra =1$. We identify $A$ and $X$ with $M\otimes_B \widetilde{M}$
and $A\otimes_A X$ by the isomorphisms defined by
\begin{align*}
m\otimes n \in M\otimes_B \widetilde{M} &  \mapsto {}_A \la m \, , \, n \ra \in A \, , \\
a\otimes x\in A\otimes_A X &  \mapsto a\cdot x\in X .
\end{align*}
Let $x\in X$, $m\in M$. For any $x\otimes m\in X\otimes_A M$,
$$
x\otimes m =1_A \cdot x\otimes m =\sum_i {}_A \la u_i \, , \, u_i \ra \cdot x\otimes m
=\sum_i u_i \otimes \widetilde{u_i}\otimes x\otimes m .
$$
Hence there is a linear map $\Psi$ from $X\otimes_A M$ to $M\otimes_B Y$ defined by
$$
\Psi(x\otimes m)=\sum_i u_i \otimes \Phi(\widetilde{u_i }\otimes x \otimes m)
$$
for any $x\in X$, $m\in M$. By the definition of $\Psi$, we can see that $\Psi$ is a Hilbert
$A-B$-bimodule isomorphism of $X\otimes_A M$ onto $M\otimes_B Y$.

\begin{lemma}\label{lem:unique} With the above notation, the Hilbert $A-B$-bimodule
isomorphism $\Psi$ of $X\otimes_A M$ onto $M\otimes_B Y$ is independent of the choice
of a finite subset $\{u_i \}$ of $M$ with $\sum_i {}_A \la u_i  \, , \, u_i \ra =1$.
\end{lemma}
\begin{proof}
Let $\{v_j \}$ be another finite subset of $M$ with $\sum_j {}_A \la v_j \, , \, v_j \ra=1$.
Then for any $x\in X$, $m\in M$,
\begin{align*}
\sum_i u_i \otimes \Phi(\widetilde{u_i}\otimes x\otimes m ) & =
\sum_{i, j}{}_A \la v_j \, , \, v_j \ra \cdot u_i \otimes \Phi(\widetilde{u_i }\otimes x \otimes m ) \\
& =\sum_{i, j}v_j \cdot \la v_j \, , \, u_i \ra_B \otimes \Phi(\widetilde{u_i}\otimes x \otimes m ) \\
& =\sum_{i, j}v_j \otimes \Phi([u_i \cdot \la u_i \, , \, v_j \ra_B ]^{\widetilde{}}\otimes x\otimes m) \\
& =\sum_j v_j \otimes\Phi(\widetilde{v_j}\otimes x\otimes m ) .
\end{align*}
Therefore, we obtain the conclusion.
\end{proof}

Similarly let $\widetilde{\Psi}$ be the Hilbert $A-B$-bimodule isomorphism of $\widetilde{X}\otimes_A M$
onto $M\otimes_B \widetilde{Y}$ defined by
$$
\widetilde{\Psi}(\widetilde{x}\otimes m)
=\sum_i u_i \otimes \widetilde{\Phi}(\widetilde{u_i}\otimes\widetilde{x}\otimes m)
$$
for any $x\in X$, $m\in M$. We compute the inverse map of $\Psi$, which is a Hilbert
$A-B$-bimodule isomorphism of $M\otimes_B Y$ onto $X\otimes_A M$.
Let $\Theta$ be the linear map from $M\otimes_B Y$ to $X\otimes_A M$ defined by
$$
\Theta(m\otimes y)=m\otimes\Phi^{-1}(y)
$$
for any $m\in M$, $y\in Y$, where we identify $M\otimes_B \widetilde{M}\otimes_A X\otimes_A M$ with
$X\otimes_A M$ as Hilbert $A-B$-bimodules by the map
$$
m\otimes\widetilde{n}\otimes x\otimes m_1 \in M\otimes_B \widetilde{M}\otimes_A X\otimes_A M
\mapsto {}_A \la m, n \ra \cdot x \otimes m_1\in X\otimes_A M .
$$

\begin{lemma}\label{lem:inverse} With the above notation, $\Theta$ is a Hilbert $A-B$-bimodule isomorphism
of $M\otimes_B Y$ onot $X\otimes_A M$ such that 
$\Theta\circ\Psi=\id_{X\otimes_A M}$ and $\Psi\circ\Theta =\id_{M \otimes_B Y}$.
\end{lemma}
\begin{proof} Let $m, m_1 \in M$, $y, y_1 \in Y$. Then
\begin{align*}
{}_A \la \Theta(m\otimes y) \, , \, \Theta(m_1 \otimes y_1 ) \ra & ={}_A \la m\otimes \Phi^{-1}(y) \, , \,
m_1 \otimes \Phi^{-1}(y_1 ) \ra \\
& = {}_A \la m\cdot {}_B \la \Phi^{-1}(y) \, , \, \Phi^{-1}(y_1  ) \ra , m_1 \ra \\
& ={}_A \la m\cdot {}_B \la y, y_1 \ra, m_1 \ra
={}_A \la m\otimes y \, , \, m_1 \otimes y_1 \ra .
\end{align*}
Hence $\Theta$ preserves the left $A$-valued inner products. Also,
\begin{align*}
\la \Theta(m\otimes y) \, ,\, \Theta(m_1 \otimes y_1 ) \ra_B & =\la m\otimes\Phi^{-1}(y) \, , \,
m_1 \otimes\Phi^{-1}(y_1 ) \ra_B \\
& =\la m, \, \la \Phi^{-1}(y) \, , \, \Phi^{-1}(y_1 ) \ra_B \cdot m_1 \ra_B \\
& =\la m \, , \, \la y , y_1 \ra_B \cdot m_1 \ra =\la m\otimes y \, , \, m_1 \otimes y_1 \ra_B .
\end{align*}
Hence $\Theta$ preserves the right $B$-valued inner products.
Furthermore, for any $x\in X$, $m\in M$,
\begin{align*}
(\Theta\circ\Psi)(x\otimes m) & =\sum_i \Theta(u_i \otimes\Phi(\widetilde{u_i}\otimes x \otimes m)) \\
& =\sum_i u_i \otimes \widetilde{u_i}\otimes x\otimes m \\
& =\sum_i {}_A \la u_i , u_i \ra\cdot x\otimes m=x\otimes m
\end{align*}
since we identify $M\otimes \widetilde{M}$ with $A$ as $A-A$-equivalence bimodules by the map
$m\otimes\widetilde{n}\in M\otimes_B \widetilde{M}\mapsto {}_A \la m, n \ra\in A$.
Hence $\Theta\circ\Psi=\id_{X\otimes_A M}$. Hence $\Psi\circ\Theta\circ\Psi=\Psi$ on $X\otimes_A M$.
Since $\Psi$ is surjective, $\Psi\circ\Theta=\id_{M\otimes_B Y}$. Therefore, by the remark after
\cite [Definition 1.1.18]{JT:KK}, $\Theta$ is a Hilbert $A-B$-bimodule isomorphism of $M\otimes_B Y$ onto
$X\otimes_A M$ such that $\Theta\circ\Psi=\id_{X\otimes_A M}$ and $\Psi\circ\Theta=\id_{M\otimes_B Y}$.
\end{proof}

Similarly, we see that the inverse map of $(\widetilde{\Psi})^{-1}$ is defined by
$$
(\widetilde{\Psi})^{-1}(m\otimes \widetilde{y})=m\otimes(\widetilde{\Phi})^{-1}(\widetilde{y})
$$
for any $m\in M$, $y\in Y$, where we identify $M\otimes_B \widetilde{M}\otimes_A \widetilde{X}\otimes_A M$
with $\widetilde{X}\otimes_A M$ as Hilbert $A-B$-bimodules by the map
$$
m\otimes\widetilde{n}\otimes\widetilde{x}\otimes m_1\in M\otimes_B
\widetilde{M}\otimes_A \widetilde{X}\otimes_A M
\mapsto {}_A \la m, n \ra\cdot \widetilde{x}\otimes m_1 \in \widetilde{X}\otimes_A M .
$$

We prepare some lemmas in order to show Proposition \ref{prop:condition}.

\begin{lemma}\label{lem:one-two} Let $A$ and $B$ be unital $C^*$-algebras.
Let $X$ and $Y$ be an involutive Hilbert $A-A$-bimodule and an involutive
Hilbert $B-B$-bimodule, respectively. Let $\mathcal{A}_X =\{A_t \}_{t\in \BZ_2}$ and
$\mathcal{A}_Y =\{B_t \}_{t\in \BZ_2}$ be $C^*$-algebraic bundles over $\BZ_2$
induced by $X$ and $Y$, respectively. We suppose that there is an $\mathcal{A}_X -\mathcal{A}_Y$-equivalence bundle $\mathcal{M}=\{M_t \}_{t\in \BZ_2}$
over $\BZ_2$ such that
$$
{}_{C} \la M_t \, , \, M_s \ra =A_{ts^{-1}} , \quad \la M_t \, , \, M_s \ra_{D}=B_{t^{-1}s}
$$
for any $t, s\in \BZ_2$, where $C=A\oplus X$ and $D=B\oplus Y$. Then
there is an $A-B$-equivalence bimodule $M$ such that
$Y\cong \widetilde{M}\otimes_A X\otimes_A M$
as involutive Hilbert $B-B$-bimodules.
\end{lemma}
\begin{proof} By the assumptions, $M_0$ is an $A-B$-equivalence bimodule.
Let $M=M_0$. Then by Lemma \ref {lem:involution}, $\widetilde{M}\otimes_A X\otimes_A M$ is
an involutive Hilbert $B-B$-bimodule whose involution is defined by
$(\widetilde{m}\otimes x\otimes n)^{\natural}=\widetilde{n}\otimes x^{\natural}\otimes m$
for any $m, n\in M$, $x\in X$. We show that $Y\cong\widetilde{M}\otimes_A X\otimes_A M$
as involutive Hilbert $B-B$-bimodules. Let $\Phi$ be the map from $\widetilde{M}\otimes_A X\otimes_A M$
to $Y$ defined by
$$
\Phi(\widetilde{m}\otimes x\otimes n) =\la m \, , \, x\cdot n \ra_D
$$
for any $m, n\in M$, $x\in X$. Since $A_1 =X$ and $M=M_0$, $X\cdot M_0 \subset M_1$.
And $\la M_0 \, , \, M_1 \ra_D \in B_1 =Y$. Hence $\Phi$ is a map from
$\widetilde{M}\otimes_A X \otimes_A M$ to $Y$. Clearly, $\Phi$ is a linear and $B-B$-bimodule map.
We show that $\Phi$ is surjective. Indeed,
$$
X\cdot M =A_1 \cdot M_0 ={}_C \la M_1 \, , \, M_0 \ra \cdot M_0
= M_1 \cdot \la M_0 \, , \, M_0 \ra_D 
=M_1 \cdot B =M_1
$$
by \cite [Proposition 1.7]{BMS:quasi}. Hence
$\la M \, , \, X\cdot M \ra_D =\la M \, , \, M_1 \ra_D =Y$.
Thus, $\Phi$ is surjective. Let $m, n, m_1 , n_1 \in M$, $x, x_1 \in X$. Then
\begin{align*}
{}_B \la \widetilde{m}\otimes x\otimes n \, , \, \widetilde{m_1}\otimes x_1 \otimes n_1 \ra 
& = {}_B \la \widetilde{m} \cdot {}_A \la x\otimes n \, , \, x_1 \otimes n_1 \ra \, , \, \widetilde{m_1} \ra \\
& ={}_B \la [\, {}_A \la x_1 \otimes n_1 \, , \, x\otimes n \ra \cdot m]^{\widetilde{}} \, , \, \widetilde{m_1} \ra \\
& =\la \, {}_A \la x_1 \otimes n_1 \, , \, x\otimes n \ra \cdot m \, , \, m_1 \ra_B \\
& =\la \, {}_A \la x_1 \cdot {}_A \la n_1 \, , \, n \ra  \, , \, x \ra \cdot m \, , \, m_1 \ra_B \\
& =\la [(x_1 \bullet \, {}_C \la n_1 \, , \, n \ra)\bullet x^{\natural}]\cdot m\, , \, m_1 \ra_B \\
& =\la [\, {}_C \la x_1 \cdot n_1 \, , \, n \ra \bullet x^{\natural}]\cdot m \, , \, m_1 \ra_B \\
& =\la \, {}_C \la [x_1 \cdot n_1 ] \, , \, n \ra \cdot [x^{\natural} \cdot m]\, , \, m_1 \ra_B \\
& =\la [x_1 \cdot n_1 ] \cdot \la n \, , \, x^{\natural}\cdot m \ra_D \, , \, m_1 \ra_B \\
& =\la x^{\natural}\cdot m \, , \, n \ra_D \bullet \la x_1 \cdot n_1 \, , \, m_1 \ra_D \\
& =\la m \, , \, x\cdot n \ra_D \bullet \la m_1 \, , \, x_1 \cdot n_1 \ra_D^{\sharp} \\
& ={}_B \la \la m\, , \, x\cdot n \ra_D \, , \, \la m_1 \, , \, x_1 \cdot n_1 \ra_D \ra \\
& ={}_B \la \Phi(\widetilde{m}\otimes x\otimes n) \, , \, \Phi(\widetilde{m_1}\otimes x_1 \otimes n_1 ) \ra .
\end{align*}
Hence $\Phi$ preserves the left $B$-valued inner products. Also,
\begin{align*}
\la \widetilde{m}\otimes x\otimes n \, , \, \widetilde{m_1}\otimes x_1 \otimes n_1 \ra_B
& =\la n \, , \la \widetilde{m}\otimes x \, , \, \widetilde{m_1}\otimes x_1 \ra_A \cdot n_1 \ra_B \\
& =\la n \, , \, \la x \, , \, \la \widetilde{m} \, , \, \widetilde{m_1} \ra_A \cdot x_1 \ra_A \cdot n_1 \ra_B \\
& =\la n \, , \, \la x \, , \, {}_A \la m \, , \, m_1 \ra \cdot x_1 \ra_A \cdot n_1 \ra_B \\
& =\la n \, , \, \la x\, , \, {}_C \la m \, , \, m_1 \ra \bullet x_1 \ra_C \cdot n_1 \ra_B \\
& =\la n \, , \, [x^{\natural}\bullet {}_C \la m \, , \, m_1 \ra\bullet x_1 ] \cdot n_1 \ra_B \\
& =\la n \, , \, [x^{\natural}\bullet {}_C \la m \, , \, m_1 \ra ]\cdot (x_1 \cdot n_1 ) \ra_B \\
& =\la n \, , \, {}_C \la x^{\natural}\cdot m \, , \, m_1 \ra \cdot (x_1 \cdot n_1 )\ra_B \\
& =\la n \, , \, [x^{\natural}\cdot m] \cdot \la m_1 \, , \, x_1 \cdot n_1 \ra_D \, \ra_B \\
& =\la n \, , \, x^{\natural}\cdot m \ra_D \bullet \la m_1 \, , \, x_1 \cdot n_1 \ra_D \\
& =\la x\cdot n \, , \, m \ra_D \bullet \la m_1 \, , \, x_1 \cdot n_1 \ra_D \\
& =\la m \, , \, x\cdot n \ra_D^{\sharp}\bullet \la m_1 \, , \, x_1 \cdot n_1 \ra_D \\
& =\la \Phi(\widetilde{m}\otimes x\otimes n) \, , \, \Phi(\widetilde{m_1}\otimes x_1 \otimes n_1 ) \ra_B .
\end{align*}
Hence $\Phi$ preserves the right $B$-valued inner products. Furthermore,
$$
\Phi(\widetilde{m}\otimes x\otimes n)^{\natural} =\la m\, , \, x\cdot n \ra_Y^{\natural}
=\la m \, , \, x\cdot n \ra_D^{\sharp}=\la x\cdot n \, , \, m \ra_D =\la x\cdot n \, ,\, m \ra_Y .
$$
On the other hand,
$$
\Phi((\widetilde{m}\otimes x\otimes n)^{\natural})=\Phi(\widetilde{n}\otimes x^{\natural}\otimes m)
=\la n \, , \, x^{\natural}\cdot m \ra_Y =\la x\cdot n \, , \, m \ra_Y =\Phi(\widetilde{m}\otimes x\otimes n)^{\natural} .
$$
Hence $\Phi$ preserves the involutions $\natural$. Therefore, $Y\cong \widetilde{M}\otimes_A X\otimes_A M$
as involutive Hilbert $B-B$-bimodules.
\end{proof}

Let $A$ and $B$ be unital $C^*$-algebras.
Let $X$ and $Y$ be an involutive Hilbert $A-A$-bimodule and an involutive
Hilbert $B-B$-bimodule, respectively. 
We suppose that there is an $A-B$-equivalence bimodule $M$ such that
$$
Y\cong \widetilde{M}\otimes_A X\otimes_A M
$$
as involutive Hilbert $B-B$-bimodules. Let $\Phi$ be an involutive Hilbert $B-B$-bimodule
isomorphism of $\widetilde{M}\otimes_A X\otimes_A M$ onto $Y$. Then by the above discussions,
there are the Hilbert $A-B$-bimodule isomorphisms $\Psi$ of $X\otimes _A M$ onto
$M\otimes_B Y$ and $\widetilde{\Psi}$ of $\widetilde{X}\otimes_A M$ onto
$M\otimes_B \widetilde{Y}$, respectively.
We construct a $C_X -C_Y$-equivalence bimodule
from $M$. Let $C_M$ be the linear span of the set
$$
{}^X \! C_M =\left\{ \begin{bmatrix}
m_1 & x\otimes m_2 \\
\widetilde{x^{\natural}}\otimes m_2 & m_1
\end{bmatrix} \,
| \, \, m_1, m_2 \in M,\, x\in X \right\} .
$$
We define the left $C_X$-action on $C_M$ by
\begin{align*}
& \begin{bmatrix}
a & z  \\
\widetilde{z^{\natural}} & a
\end{bmatrix}
\cdot \begin{bmatrix}
m_1 & x\otimes m_2 \\
\widetilde{x^{\natural}}\otimes m_2 & m_1
\end{bmatrix} \\
& =\begin{bmatrix}
a\otimes m_1 +z\otimes \widetilde{x^{\natural}}\otimes m_2 & a\otimes x\otimes m_2 +z\otimes m_1 \\
\widetilde{z^{\natural}}\otimes m_1 +a\otimes\widetilde{x^{\natural}}\otimes m_2 &
\widetilde{z^{\natural}}\otimes x\otimes m_2 +a\otimes m_1
\end{bmatrix}
\end{align*}
for any $a\in A$, $m_1 , m_2 \in M$, $x, z\in X$, where we regard the tensor product as a left $C_X$-action
on $C_M$ in the
formal manner. But we identify $A\otimes_A M$ and $X\otimes_A \widetilde{X}$, $\widetilde{X}\otimes_A X$
with $M$ and closed two-sided ideals of $A$ by the isomorphism and the monomorphisms defined by
\begin{align*}
a\otimes m\in A\otimes_A M & \mapsto a\cdot m\in M , \\
x\otimes \widetilde{z}\in X\otimes_A \widetilde{X} & \mapsto {}_A \la x, z \ra\in A , \\
\widetilde{x}\otimes z\in \widetilde{X}\otimes_A X & \mapsto \la x, z \ra_A \in A .
\end{align*}
Hence we obtain that
\begin{align*}
& \begin{bmatrix}
a & z  \\
\widetilde{z^{\natural}} & a
\end{bmatrix}
\cdot \begin{bmatrix}
m_1 & x\otimes m_2 \\
\widetilde{x^{\natural}}\otimes m_2 & m_1
\end{bmatrix} \\
& =\begin{bmatrix}
a\cdot m_1 +{}_A \la z\ \, , \, {x^{\natural}}\ra \cdot m_2 &  a\cdot  x\otimes m_2 +z\otimes m_1 \\
\widetilde{z^{\natural}}\otimes m_1 +\widetilde{(a\cdot x)^{\natural}}\otimes m_2 &
\la \widetilde{z^{\natural}} \, , \,  x \ra_A \cdot m_2 +a\cdot m_1
\end{bmatrix}\in C_M .
\end{align*}
We define the right $C_Y$-action on $C_M$ by
\begin{align*}
& \begin{bmatrix} m_1 & x\otimes m_2 \\
\widetilde{x^{\natural}}\otimes m_2 & m_1 \end{bmatrix}\cdot
\begin{bmatrix} b & y \\
\widetilde{y^{\natural}} & b \end{bmatrix} \\
& =\begin{bmatrix} m_1 \otimes b+x\otimes m_2 \otimes \widetilde{y^{\natural}} & m_1
\otimes y+x\otimes m_2\otimes b \\
\widetilde{x^{\natural}}\otimes m_2 \otimes b+m_1 \widetilde{y^{\natural}} & \widetilde{x^{\natural}}\otimes m_2
\otimes y+m_1 \otimes b \end{bmatrix}
\end{align*}
for any $b\in B$, $x\in X$, $y\in Y$, $m_1 , m_2 \in M$, where we regard the tensor product as a right
$C_Y$-action on
$C_M$ in the formal manner. But we identify $X\otimes_A M$ and $\widetilde{X}\otimes_A M$ with
$M\otimes_B Y$ and $M\otimes_B \widetilde{Y}$ by $\Psi$ and $\widetilde{\Psi}$, respectively.
Hence we obtain that
\begin{align*}
& \begin{bmatrix} m_1 & x\otimes m_2 \\
\widetilde{x^{\natural}}\otimes m_2 & m_1 \end{bmatrix}\cdot
\begin{bmatrix} b & y \\
\widetilde{y^{\natural}} & b \end{bmatrix} \\
& =\begin{bmatrix} m_1 \otimes b+x\otimes (\widetilde{\Psi})^{-1}(m_2 \otimes \widetilde{y^{\natural}}) &
\Psi^{-1}(m_1\otimes y)+x\otimes m_2\otimes b \\
\widetilde{x^{\natural}}\otimes m_2 \otimes b+(\widetilde{\Psi})^{-1}(m_1 \otimes\widetilde{y^{\natural}}) &
\widetilde{x^{\natural}}\otimes \Psi^{-1}(m_2 \otimes y)+m_1 \otimes b \end{bmatrix} .
\end{align*}
Furthermore, we identify $M\otimes_B B$ and $Y\otimes_B \widetilde{Y}$, $\widetilde{Y}\otimes_B Y$
with $M$ and closed two-sided ideals of $B$ by the isomorphism and the monomorphisms defined by
\begin{align*}
m\otimes b\in M\otimes_B B & \mapsto m\cdot b\in M , \\
y\otimes\widetilde{z}\in Y\otimes_B \widetilde{Y} & \mapsto {}_B \la y, z \ra\in B , \\
\widetilde{y}\otimes z\in \widetilde{Y}\otimes_B Y & \mapsto \la y, z \ra_B ,
\end{align*}
respectively. Then
$$
x\otimes(\widetilde{\Psi})^{-1}(m_2 \otimes y^{\natural})
=\widetilde{x^{\natural}}\otimes\Psi^{-1}(m_2 \otimes y) , \quad
\begin{bmatrix} m_ 1 & x\otimes m_2 \\
\widetilde{x}^{\natural} \otimes m_2 & m_1 \end{bmatrix}
=\begin{bmatrix} b & y \\
\widetilde{y}^{\natural} & b \end{bmatrix}\in C_M .
$$
Indeed, for any $\epsilon >0$, there finite sets $\{n_k \}, \{l_k \}\subset M$
and $\{z_k \}\subset X$ such that
$$
||\Phi^{-1}(y)-\sum_k \widetilde{n_k}\otimes z_k \otimes l_k ||<\epsilon .
$$
Also,
\begin{align*}
||(\widetilde{\Phi})^{-1}(\widetilde{y}^{\natural})-
[(\sum_k \widetilde{n_k}\otimes z_k \otimes l_k )^{\natural}]^{\widetilde{}}||
& =||[\Phi^{-1}(y)^{\natural}]^{\widetilde{}}-
[(\sum_k \widetilde{n_k}\otimes z_k \otimes l_k )^{\natural}]^{\widetilde{}}|| \\
& =||\Phi^{-1}(y)-\sum_k \widetilde{n_k}\otimes z_k \otimes l_k ||<\epsilon .
\end{align*}
Thus
\begin{align*}
& ||x\otimes(\widetilde{\Psi})^{-1}(m_2 \otimes\widetilde{y^{\natural}})
-x\otimes m_2 \otimes[(\sum_k \widetilde{n_k }\otimes z_k \otimes l_k )^{\natural}]^{\widetilde{}}|| \\
& =||x\otimes m_2 \otimes (\widetilde{\Phi})^{-1}(\widetilde{y}^{\natural})-x\otimes m_2 \otimes
\sum_k \widetilde{n_k}\otimes \widetilde{z_k}^{\natural}\otimes l_k || \\
& \leq ||x|| \, ||m_1||\epsilon
\end{align*}
and
\begin{align*}
& ||\widetilde{x}^{\natural}\otimes\Psi^{-1}(m_2 \otimes y)
-\widetilde{x}^{\natural}\otimes m_2 \otimes\sum_k \widetilde{n_k }\otimes z_k \otimes l_k || \\
& =||\widetilde{x}^{\natural}\otimes m_2 \otimes (\widetilde{\Phi})^{-1}(y)
-\widetilde{x}^{\natural}\otimes m_2 \otimes\sum_k \widetilde{n_k}\otimes \widetilde{z_k}^{\natural}\otimes l_k || \\
& \leq ||x|| \, ||m_1||\epsilon .
\end{align*}
Furthermore,
\begin{align*}
x\otimes m_2 \otimes[(\sum_k \widetilde{n_k}\otimes\widetilde{z_k}\otimes l_k )^{\natural}]^{\widetilde{}}
& =\sum_k x\otimes m_2 \otimes\widetilde{n_k}\otimes\widetilde{z_k}^{\natural}\otimes l_k \\
& =\sum_k x\cdot {}_A \la m_2 \, , \, n_k \ra\otimes\widetilde{z_k}^{\natural}\otimes l_k \\
& =\sum_k {}_A  \la x\cdot {}_A \la m_2 \, , \, n_k \ra \, , z_k^{\natural} \ra \otimes l_k \\
& =\sum_k {}_A \la x\cdot {}_A \la m_2 \, , \, n_k \ra \, , z_k^{\natural} \ra \cdot l_k 
\end{align*}
and
\begin{align*}
\widetilde{x^{\natural}}\otimes m_2 \widetilde{n_k}\otimes z_k \otimes l_k
& =\sum_k \widetilde{x^{\natural}}\otimes m_2 \otimes \widetilde{n_k}\otimes z_k \otimes l_k \\
& =\sum_k \widetilde{x^{\natural}}\otimes {}_A \la m_2 \, , \, n_k \ra \cdot z_k \otimes l_k \\
& =\sum_k \la x^{\natural} \, , \, {}_A \la m_2 \, , \, n_k \ra  \cdot z_k \ra_A \otimes l_k \\
& =\sum_k {}_A \la x \, , \, ({}_A \la m_2 \, , \, n_k \ra \cdot z_k )^{\natural} \ra \cdot l_k \\
& =\sum_k {}_A \la x \, , \, z_k^{\natural}\cdot {}_A \la n_k \, , \, m_2 \ra \ra\cdot l_k \\
& =\sum_k {}_A \la x\cdot {}_A \la m_2 \, , \, n_k \ra \, , z_k^{\natural} \ra \cdot l_k ,
\end{align*}
where we identify $A\otimes_A M$ and $X\otimes_A \widetilde{X}$, $\widetilde{X}\otimes_A X$
with $M$ and closed two-sided ideals of $A$ by the isomorphism and the monomorphisms defined by
\begin{align*}
a\otimes m \in A\otimes_A M & \mapsto a\cdot m\in M , \\
x\otimes\widetilde{z} \in X\otimes_A \widetilde{X} & \mapsto {}_A \la x, z \ra\in A ,\\
\widetilde{x}\otimes z\in \widetilde{X}\otimes_A X & \mapsto \la x, z \ra_A \in A .
\end{align*}
Hence
$$
x\otimes m_2 \otimes[(\sum_k \widetilde{n_k}\otimes\widetilde{z_k}\otimes l_k )^{\natural}]^{\widetilde{}}
\widetilde{x^{\natural}}\otimes m_2 \widetilde{n_k}\otimes z_k \otimes l_k .
$$
It follows that
$$
||x\otimes(\widetilde{\Psi})^{-1}(m_2 \otimes y^{\natural})-\widetilde{}^{\natural}\otimes\Psi^{-1}(m2 \otimes y)||
\leq 2||x|| \, ||m_2 ||\epsilon .
$$
Since $\epsilon$ is arbitrary, we obtain that
$$
x\otimes(\widetilde{\Psi})^{-1}(m_2 \otimes y^{\natural})
=\widetilde{x^{\natural}}\otimes\Psi^{-1}(m_2 \otimes y) , \quad
\begin{bmatrix}
m_1 & x\otimes m_2 \\
\widetilde{x^{\natural}}\otimes m_2 & m_1 \end{bmatrix}\cdot
\begin{bmatrix} b & y \\
\widetilde{y^{\natural}} & b \end{bmatrix}\in C_M .
$$
Before we define a left $C_X$-valued inner product and a right $C_Y$-valued inner product on $C_M$,
we define a conjugate linear map on $C_M$,
$$
\begin{bmatrix} m_1 & x\otimes m_2 \\
\widetilde{x^{\natural}} \otimes m_2 & m_1 \end{bmatrix}\in C_M \mapsto
\begin{bmatrix} m_1 & x\otimes m_2 \\
\widetilde{x^{\natural}} \otimes m_2 & m_1 \end{bmatrix}^{\widetilde{}}\in C_M
$$
by
$$
\begin{bmatrix} m_1 & x\otimes m_2 \\
\widetilde{x^{\natural}} \otimes m_2 & m_1 \end{bmatrix}^{\widetilde{}}
=\begin{bmatrix} \widetilde{m_1} & (\widetilde{x^{\natural}}\otimes m_2 )^{\widetilde{}} \\
(x\otimes m_2 )^{\widetilde{}} & \widetilde{m_1} \end{bmatrix}
$$
for any $m_1, m_2 \in M$, $x\in X$. Since we identify $\widetilde{X\otimes_A M}$
and $\widetilde{\widetilde{X}\otimes_A M}$ with $\widetilde{M}\otimes_A \widetilde{X}$ and
$\widetilde{M}\otimes_A X$ by Lemma \ref{lem:tilde}, respectively,
we obtain that
$$
\begin{bmatrix} m_1 & x\otimes m_2 \\
\widetilde{x^{\natural}} \otimes m_2 & m_1 \end{bmatrix}^{\widetilde{}}
=\begin{bmatrix} \widetilde{m_1} & \widetilde{m_2}\otimes x^{\natural} \\
\widetilde{m_2}\otimes\widetilde{x} & \widetilde{m_1} \end{bmatrix} .
$$
We define the left $C_X$-valued inner product on $C_M$ by
\begin{align*}
& {}_{C_X} \la \begin{bmatrix} m_1 & x\otimes m_2 \\
\widetilde{x^{\natural}}\otimes m_2 & m_1 \end{bmatrix} \, , \,
\begin{bmatrix} n_1 & z\otimes n_2 \\
\widetilde{z^{\natural}}\otimes n_2 & n_1 \end{bmatrix} \ra \\
& =\begin{bmatrix} m_1 & x\otimes m_2 \\
\widetilde{x^{\natural}}\otimes m_2 & m_1 \end{bmatrix}\cdot
\begin{bmatrix} n_1 & z\otimes n_2 \\
\widetilde{z^{\natural}}\otimes n_2 & n_1 \end{bmatrix}^{\widetilde{}} \\
& =\begin{bmatrix} m_1 & x\otimes m_2 \\
\widetilde{x^{\natural}}\otimes m_2 & m_1 \end{bmatrix}\cdot
\begin{bmatrix} \widetilde{n_1} & \widetilde{n_2}\otimes z^{\natural} \\
\widetilde{n_2}\otimes \widetilde{z} & \widetilde{n_1} \end{bmatrix} \\
& =\begin{bmatrix} m_1 \otimes\widetilde{n_1}+x\otimes m_2 \otimes \widetilde{n_2}\otimes \widetilde{z} &
m_1 \otimes \widetilde{n_2}\otimes z^{\natural}+x\otimes m_2 \otimes \widetilde{n_1} \\
\widetilde{x^{\natural}}\otimes m_2 \otimes \widetilde{n_1}+m_1 \otimes \widetilde{n_2}\otimes \widetilde{z} &
\widetilde{x^{\natural}}\otimes m_2 \otimes \widetilde{n_2}\otimes z^{\natural}+m_1 \otimes \widetilde{n_1}
\end{bmatrix}
\end{align*}
for any $m_1 , m_2 , n_1 , n_2 \in M$, $x, z\in X$, where we regard the tensor product as a product in $C_M$ in the
formal manner. Identifying in the same way as above,
\begin{align*}
& {}_{C_X} \la \begin{bmatrix} m_1 & x\otimes m_2 \\
\widetilde{x^{\natural}}\otimes m_2 & m_1 \end{bmatrix} \, , \,
\begin{bmatrix} n_1 & z\otimes n_2 \\
\widetilde{z^{\natural}}\otimes n_2 & n_1 \end{bmatrix} \ra \\
& =\begin{bmatrix} {}_A \la m_1, n_1 \ra+{}_A \la x\cdot {}_A \la m_2, n_2 \ra \, ,z \ra &
{}_A \la m_1 , n_2 \ra , \cdot z^{\natural}+x\cdot {}_A \la m_2 , n_1 \ra \\
\widetilde{x^{\natural}}\cdot {}_A \la m_2 , n_1 \ra +{}_A \la m_1 , n_2 \ra \cdot\widetilde{z} &
{}_A \la x\cdot {}_A \la m_2 , n_2 \ra \, , z\ra +{}_A \la m_1 , n_1 \ra \end{bmatrix} .
\end{align*}
We define the right $C_Y$-valued inner product on $C_M$ by
\begin{align*}
& \la \begin{bmatrix} m_1 & x\otimes m_2 \\
\widetilde{x^{\natural}}\otimes m_2 & m_1 \end{bmatrix} \, , \,
\begin{bmatrix} n_1 & z\otimes n_2 \\
\widetilde{z^{\natural}}\otimes n_2 & n_1 \end{bmatrix} \ra_{C_Y} \\
& =\begin{bmatrix} m_1 & x\otimes m_2 \\
\widetilde{x^{\natural}}\otimes m_2 & m_1 \end{bmatrix}^{\widetilde{}}\cdot
\begin{bmatrix} n_1 & z\otimes n_2 \\
\widetilde{z^{\natural}}\otimes n_2 & n_1 \end{bmatrix} \\
& =\begin{bmatrix} \widetilde{m_1} & \widetilde{m_2}\otimes x^{\natural} \\
\widetilde{m_2}\otimes\widetilde{x} & \widetilde{m_1} \end{bmatrix}\cdot
\begin{bmatrix} n_1 & z\otimes n_2 \\
\widetilde{z^{\natural}}\otimes n_2 & n_1 \end{bmatrix} \\
& =\begin{bmatrix} \widetilde{m_1}\otimes n_1
+\widetilde{m_2}\otimes x^{\natural}\otimes\widetilde{z^{\natural}}\otimes n_2 &
\widetilde{m_1}\otimes z \otimes n_2 +\widetilde{m_2}\otimes x^{\natural}\otimes n_1 \\
\widetilde{m_2}\otimes\widetilde{x}\otimes n_1 +\widetilde{m_1}\otimes\widetilde{z^{\natural}}\otimes n_2 & 
\widetilde{m_2}\otimes\widetilde{x}\otimes z \otimes n_2 +\widetilde{m_1}\otimes n_1 \end{bmatrix}
\end{align*}
for any $m_1, m_2 , n_1 , n_2 \in M$, $x, z\in X$, where we regard the tensor product as
a product in $C_M$ in the formal manner. Identifying in the same way as above and by the
isomorphism $\Psi$ and $\widetilde{\Psi}$,
\begin{align*}
& \la \begin{bmatrix} m_1 & x\otimes m_2 \\
\widetilde{x^{\natural}}\otimes m_2 & m_1 \end{bmatrix} \, , \,
\begin{bmatrix} n_1 & z\otimes n_2 \\
\widetilde{z^{\natural}}\otimes n_2 & n_1 \end{bmatrix} \ra_{C_Y} \\
& =\begin{bmatrix}\la m_1 , n_1 \ra_B +\la m_2 , \la x, z \ra_A \cdot n_2 \ra_B &
\widetilde{m_1}\otimes\Psi(z\otimes n_2 )+\widetilde{m_2}\otimes\Psi(x^{\natural}\otimes n_1) \\
\widetilde{m_2}\otimes\widetilde{\Psi}(\widetilde{x}\otimes n_1 )
+\widetilde{m_1}\otimes\widetilde{\Psi}(\widetilde{z^{\natural}}\otimes n_2) &
\la m_2 , \la x, z \ra_A \cdot n_2 \ra_B +\la m_1 , n_1 \ra_B \end{bmatrix} .
\end{align*}
Here,
\begin{align*}
\widetilde{m_1}\otimes\Psi(z\otimes n_2 ) &=\sum_i \widetilde{m_1}\otimes u_i \otimes\Phi(\widetilde{u_i}
\otimes z\otimes n_2 ) \\
& =\sum_i \la m_1 , u_i \ra_B \cdot \Phi(\widetilde{u_i}\otimes z\otimes n_2 ) \\
& =\sum_i \Phi(\la m_1 , u_i \ra_B \cdot\widetilde{u_i}\otimes z\otimes n_2 ) \\
& =\sum_i \Phi([u_i \cdot \la u_i , m_1 \ra_B ]^{\widetilde{}}\otimes z\otimes n_2 ) \\
& =\sum_i \Phi([{}_A \la u_i , u_i \ra \cdot m_1 ]^{\widetilde{}}\otimes z\otimes n_2 ) \\
& =\Phi(\widetilde{m_1}\otimes z\otimes n_2 )\in Y , \\
\widetilde{m_2}\otimes\Psi(x^{\natural}\otimes n_1 ) & =\Phi(\widetilde{m_2}\otimes x^{\natural}\otimes n_1 )\in Y
\end{align*}
Also,
\begin{align*}
\widetilde{m_2}\otimes\widetilde{\Psi}(\widetilde{x}\otimes n_1 ) & =\sum_i \widetilde{m_2}
\otimes u_i \otimes\widetilde{\Phi}(\widetilde{u_i}\otimes\widetilde{x}\otimes n_1 ) \\
& =\sum_i \la m_2 , u_i \ra_B\cdot \Phi(\widetilde{n_1}\otimes x\otimes u_i )^{\widetilde{}} \\
& =\sum_i \Phi(\widetilde{n_1}\otimes x \otimes u_i \cdot \la u_i , m_2 \ra_B )^{\widetilde{}} \\
& =\sum_i \Phi(\widetilde{n_1}\otimes x \otimes {}_A \la u_i , u_i \ra \cdot m_2 )^{\widetilde{}} \\
& =\Phi (\widetilde{n_1}\otimes x \otimes m_2 )^{\widetilde{}}\in\widetilde{Y} , \\
\widetilde{n_1}\otimes \widetilde{\Phi}(\widetilde{z^{\natural}}\otimes n_2 ) & =\sum_i
\widetilde{m_1}\otimes u_i \otimes\widetilde{\Phi}
(\widetilde{u_i }\otimes \widetilde{z^{\natural}}\otimes n_2 ) \\
& =\sum_i \la m_1 , u_i \ra_B \cdot 
\Phi(\widetilde{n_2}\otimes z^{\natural}\otimes u_i )^{\widetilde{}} \\
& =\sum_i \Phi (\widetilde{n_2}\otimes z^{\natural}\otimes m_1 )^{\widetilde{}}\in\widetilde{Y} .
\end{align*}
Thus
\begin{align*}
[\widetilde{m_2}\otimes\widetilde{\Psi}(\widetilde{x}\otimes n_1 )
+\widetilde{n_1}\otimes\widetilde{\Psi}
(\widetilde{z^{\natural}}\otimes n_2 )]^{\widetilde{\natural}}
& =\Phi(\widetilde{n_1}\otimes x \otimes m_2 )^{\natural}+\Phi(\widetilde{n_2}\otimes z^{\natural}
\otimes m_1 )^{\natural} \\
& =\Phi(\widetilde{m_2}\otimes x^{\natural}\otimes n_1 )
+\Phi(\widetilde{m_1}\otimes z\otimes n_2 ) \\
& =\widetilde{m_1}\otimes\Psi(z\otimes x)+\widetilde{m_2}\otimes\Psi(x^{\natural}\otimes n_1 ) .
\end{align*}
Hence
$$
\la \begin{bmatrix} m_1 & x\otimes m_2 \\
\widetilde{x^{\natural}}\otimes m_2 & m_1 \end{bmatrix} \, , \,
\begin{bmatrix} n_1 & z\otimes n_2 \\
\widetilde{z^{\natural}}\otimes n_2 & n_1 \end{bmatrix} \ra_{C_Y} \in C_Y .
$$
By the above definitions $C_M$ has the left $C_X$-and the right $C_Y$-actions and the left
$C_X$-valued inner product and the right $C_Y$-inner product.
\par
Let $C_M '$ be the linear span of the set
$$
C_M^Y =\{ \begin{bmatrix} m_1 & m_2 \otimes y \\
m_2 \otimes \widetilde{y^{\natural}} & m_1 \end{bmatrix} \, | \, m_1, m_2 \in M , \, y\in Y \} .
$$
In the similar way to the above, we define a left $C_X$-and a right $C_Y$-actions on $C_M '$
and a left $C_X$-valued inner product and a right $C_Y$-valued inner product.
But identifying $X\otimes_A M$ and $\widetilde{X}\otimes_A M$ with $M\otimes_B Y$ and
$M\otimes_B \widetilde{Y}$ by $\Psi$ and $\widetilde{\Psi}$, respectively,
we can see that each of them coincides with the other by routine computations. For example,
we show that the right $C_Y$-actions on $C_M$ and $C_M '$ coincide by $\Psi$ and
$\widetilde{\Psi}$. Indeed, for any $m_1 , m_2 \in M$, $x\in X$, $b\in B$, $y\in Y$,
\begin{align*}
& \begin{bmatrix} m_1 & x\otimes m_2 \\
\widetilde{x^{\natural}}\otimes m_2 & m_1 \end{bmatrix}\cdot
\begin{bmatrix}b & y \\
\widetilde{y^{\natural}} & b \end{bmatrix} \\
& =\begin{bmatrix} m_1 \cdot b+x\otimes m_2 \otimes \widetilde{y^{\natural}} &
m_1 \otimes y+x\otimes m_2 \cdot b \\
\widetilde{x^{\natural}}\otimes m_2 \cdot b+m_1 \otimes\widetilde{y^{\natural}} &
\widetilde{x^{\natural}}\otimes m_2 \otimes y+m_1 \cdot b \end{bmatrix} .
\end{align*}
Regarding elements in $X\otimes_A M$ and $\widetilde{X}\otimes_A M$ as elements
in $M\otimes_B Y$ and $M\otimes_B \widetilde{Y}$ by the isomorphisms $\Psi$ and
$\widetilde{\Psi}$ defined as above, respectively,
\begin{align*}
& \begin{bmatrix} m_1 & x\otimes m_2 \\
\widetilde{x^{\natural}}\otimes m_2 & m_1 \end{bmatrix} \cdot
\begin{bmatrix} b & y \\
\widetilde{y^{\natural}} & b \end{bmatrix} \\
& =\begin{bmatrix} m_1 \cdot b+\sum_i u_i \otimes\Phi(\widetilde{u_i}\otimes x\otimes m_2 )
\otimes\widetilde{y^{\natural}} &
m_1 \otimes y+\sum_i \widetilde{u_i}\otimes\Phi(u_i \otimes x\otimes m_2 \cdot b) \\
\sum_i u_i \otimes \widetilde{\Phi}(\widetilde{u_i}\otimes \widetilde{x^{\natural}}\otimes m_2 \cdot b)
+m_1 \otimes\widetilde{y^{\natural}} &
\sum_i u_i\otimes\widetilde{\Phi}(\widetilde{u_i} \otimes\widetilde{x^{\natural}}\otimes m_2 )\otimes y
+m_1 \cdot b \end{bmatrix} \\
& =\begin{bmatrix} m_1 \cdot b+\sum_i u_i \, \otimes {}_B \la \Phi(\widetilde{u_i}\otimes x\otimes m_2) \, , \,
y^{\natural} \ra & 
m_1 \otimes y+\sum_i \widetilde{u_i}\otimes\Phi(u_i \otimes x\otimes m_2 )\cdot b \\
\sum_i u_i \otimes\widetilde{\Phi}(\widetilde{u_i}\otimes\widetilde{x^{\natural}}\otimes m_2 )\cdot b
+m_1 \otimes\widetilde{y^{\natural}} &
\sum_i u_i \otimes \la \Phi(\widetilde{m_2}\otimes x^{\natural}\otimes u_i ) \, , \,
y \ra_B +m_1 \cdot b \end{bmatrix} .
\end{align*}
On the other hand,
\begin{align*}
& \begin{bmatrix} m_1 & x\otimes m_2 \\
\widetilde{x^{\natural}}\otimes m_2 & m_1 \end{bmatrix} \cdot
\begin{bmatrix} b & y \\
\widetilde{y^{\natural}} & b \end{bmatrix} \\
& =\begin{bmatrix}
m_1 & \sum_i u_i \otimes\Phi(\widetilde{u_i}\otimes x\otimes m_2 ) \\
\sum_i u_i \otimes \widetilde{\Phi}(\widetilde{u_i}\otimes\widetilde{x^{\natural}}\otimes m_2 ) & m_1 \end{bmatrix}
\cdot \begin{bmatrix} b & y \\
\widetilde{y^{\natural}} & b \end{bmatrix} \\
& =\begin{bmatrix} m_1 \cdot b+\sum_i u_i \, \otimes {}_B \la \Phi(\widetilde{u_i}\otimes x\otimes m_2) \, , \,
y^{\natural} \ra & 
m_1 \otimes y+\sum_i \widetilde{u_i}\otimes\Phi(u_i \otimes x\otimes m_2 )\cdot b \\
\sum_i u_i \otimes\widetilde{\Phi}(\widetilde{u_i}\otimes\widetilde{x^{\natural}}\otimes m_2 )\cdot b
+m_1 \otimes\widetilde{y^{\natural}} &
\sum_i u_i \otimes \la \Phi(\widetilde{m_2}\otimes x^{\natural}\otimes u_i ) \, , \,
y \ra_B +m_1 \cdot b \end{bmatrix} .
\end{align*}
Hence the right $C_Y$-actions on $C_M$ and $C_M '$ coincide.
Similarly, we can see that the left $C_X$-actions on $C_M$ and $C_M '$ coincide.
Also, we can see that the left $C_X$-valued inner products on $C_M$ and $C_M '$
coincide. Indeed,
let $x, z\in X$, $m_1 , m_2 , n_1 , n_2\in M$.
Then
\begin{align*}
& {}_{C_X} \la \begin{bmatrix} m_1 & x\otimes m_2 \\
\widetilde{x^{\natural}}\otimes m_2 & m_1 \end{bmatrix}\, , \,
\begin{bmatrix} n_1 & z\otimes n_2 \\
\widetilde{z^{\natural}}\otimes n_2 & n_1 \end{bmatrix} \ra \\
& =\begin{bmatrix} m_1 & x\otimes m_2 \\
\widetilde{x^{\natural}}\otimes m_2 & m_1 \end{bmatrix}
\begin{bmatrix} n_1 & z\otimes n_2 \\
\widetilde{z^{\natural}}\otimes n_2 & n_1 \end{bmatrix}^{\widetilde{}}
=\begin{bmatrix} m_1 & x\otimes m_2 \\
\widetilde{x^{\natural}}\otimes m_2 & m_1 \end{bmatrix}
\begin{bmatrix} \widetilde{n_1} & \widetilde{n_2}\otimes z^{\natural} \\
\widetilde{n_2}\otimes \widetilde{z} & \widetilde{n_1} \end{bmatrix} \\
& =\begin{bmatrix} {}_A \la m_1 \, , \, n_1 \ra +x\otimes m_2 \otimes \widetilde{n_2}\otimes\widetilde{z} &
{}_A \la m_1 \, , \, n_2 \ra\otimes z^{\natural}+x\otimes \, {}_A \la m_2 \, , \, n_1 \ra \\
\widetilde{x^{\natural}}\otimes \, {}_A \la m_2 \, , \, n_1 \ra +\, {}_A \la m_1 \, , \, n_2 \ra \otimes\widetilde{z} &
\widetilde{x^{\natural}}\otimes\, {}_A \la m_2 \, , \, n_2 \ra \otimes z^{\natural}+\, {}_A \la m_2 \, , \, n_1 \ra 
\end{bmatrix} \\
& =\begin{bmatrix}{}_A \la m_1 \, , \, n_1 \ra+x\otimes\, {}_A \la m_2 \, , \, n_2 \ra \otimes\widetilde{z} &
(z\cdot {}_A  \la n_2 \, , \, m_1 \ra )^{\natural}+x\cdot {}_A \la m_2 \, , \, n_1 \ra \\
[(x\cdot {}_A \la m_2 \, , \, n_1 \ra )^{\natural}]^{\widetilde{}}+(z\cdot {}_A \la n_2 \, , \, m_1 \ra )^{\widetilde{}} &
\la x^{\natural} \, , \, (z\cdot {}_A \la n_2 \, , \, m_2 \ra )^{\natural}\ra_A +\, {}_A \la m_1 \, , \, n_1 \ra \end{bmatrix} \\
& =\begin{bmatrix} {}_A \la m_1 \, , \, n_1 \ra+\, {}_A \la x\cdot  {}_A \la m_2 \, , \, n_2 \ra \, , \, z \ra &
(z\cdot  {}_A \la n_2 \, , \, m_1 \ra )^{\natural}+x\cdot  {}_A \la m_2 \, , \, n_1 \ra \\
[(x\cdot  {}_A \la m_2 \, , \, n_1 \ra )^{\natural}]^{\widetilde{}}+(z\cdot  {}_A \la n_2 \, , \, m_1 \ra)^{\widetilde{}} &
\la x^{\natural} , \, \, (z\cdot {}_A \la n_2 \, , \, m_2 \ra)^{\natural}\ra_A +\, {}_A \la m_1 \, , \, n_1 \ra \end{bmatrix} \\
& =\begin{bmatrix} {}_A \la m_1 \ , \, n_1 \ra +{}_A \la x\otimes m_2 \, , \, z\otimes n_2 \ra &
(z\cdot {}_A \la n_2 \, , \, m_1 \ra )^{\natural} +x\cdot {}_A \la m_2\, , \, n_1 \ra \\
[(x\cdot {}_A \la m_2 \, , \, n_1 \ra )^{\natural}]^{\widetilde{}}+(z\cdot {}_A \la n_2 \, , \, m_1 \ra )^{\widetilde{}} &
{}_A \la x \, , \, z\cdot {}_A \la n_2 \, , \, m_2 \ra \ra +{}_A \la m_1 \, , \, n_1 \ra \end{bmatrix} \\
& =\begin{bmatrix} {}_A \la m_1 \, , \, n_1 \ra +{}_A \la x\otimes m_2 \, , \, z\otimes n_2 \ra &
(z\cdot {}_A \la n_2 \, , \, m_1 \ra)^{\natural}+x\cdot {}_A \la m_2 \, , \, n_1 \ra \\
[(x\cdot {}_A \la m_2 \, , \, n_1 \ra )^{\natural}]^{\widetilde{}}+(z\cdot {}_A \la n_2 \, , \, m_1 \ra )^{\widetilde{}} &
{}_A \la x\otimes m_2 \, , \, z\otimes n_2 \ra +{}_A \la m_1 \, , \, n_1 \ra \end{bmatrix} .
\end{align*}
On the other hand,
\begin{align*}
& {}_{C_X} \la \begin{bmatrix} m_1 & \Psi(x\otimes m_2 ) \\
\widetilde{\Psi}(\widetilde{x^{\natural}}\otimes m_2 ) & m_1 \end{bmatrix} \, , \,
\begin{bmatrix} n_1 & \Psi(z\otimes n_2 ) \\
\widetilde{\Psi}(\widetilde{z^{\natural}}\otimes n_2 ) & n_1 \end{bmatrix} \ra \\
& =\begin{bmatrix} m_1 & \Psi(x\otimes m_2 ) \\
\widetilde{\Psi}(\widetilde{x^{\natural}}\otimes m_2 ) & m_1 \end{bmatrix}
\begin{bmatrix} \widetilde{n_1} & \widetilde{\Psi}(\widetilde{z^{\natural}}\otimes n_2 )^{\widetilde{}\,\,} \\
\Psi(z\otimes n_2 )^{\widetilde{}} & \widetilde{n_1} \end{bmatrix} \\
& =\begin{bmatrix} {}_A \la m_1 \, , \, n_1 \ra +{}_A \la \Psi(x\otimes m_2 ) \, , \, \Psi(z\otimes n_2 )\ra &
m_1 \otimes \widetilde{\Psi}(\widetilde{z^{\natural}}\otimes n_2 )^{\widetilde{}}+
\Psi(x\otimes m_2 )\otimes \widetilde{n_1} \\
\widetilde{\Psi}(\widetilde{x^{\natural}}\otimes m_2 )\otimes \widetilde{n_1}
+m_1 \otimes \Psi(y\otimes n_2 )^{\widetilde{}} &
{}_A \la \widetilde{\Psi}(\widetilde{x^{\natural}}\otimes m_2 ) \, , \, \widetilde{\Psi}(\widetilde{z^{\natural}}\otimes n_2 )\ra
+{}_A \la m_1 \, , \, n_1 \ra \end{bmatrix} .
\end{align*}
Here,
$$
{}_A \la \Psi(x\otimes m_2 ) \, , \, \Psi(z\otimes n_2 )\ra =\, {}_A \la x\otimes m_2 \, , \, z\otimes n_2 \ra .
$$
Also,
\begin{align*}
& m_1 \otimes\widetilde{\Psi}(\widetilde{z^{\natural}}\otimes n_2 )^{\widetilde{}}
+\Psi(x\otimes m_2 )\otimes \widetilde{n_1} \\
& =m_1 \otimes \sum_i [u_i \otimes \widetilde{\Phi}(\widetilde{u_i}\otimes\widetilde{z^{\natural}}
\otimes n_2 )]^{\widetilde{}}
+\sum_i u_i \otimes\Phi(\widetilde{u_i}\otimes x\otimes m_2 )\otimes\widetilde{n_1} \\
& =m_1 \otimes\sum_i \Phi(\widetilde{n_2}\otimes z^{\natural}\otimes u_i )\otimes\widetilde{u_i}+
\sum_i u_i \otimes\Phi(\widetilde{u_i}\otimes x\otimes m_2 )\otimes \widetilde{n_1} .
\end{align*}
Since we identify $\widetilde{M}\otimes_A X\otimes_A M$ with $Y$ by the involutive Hilbert $B-B$-bimodule 
isomorphism $\Phi$,
\begin{align*}
& m_1 \otimes \sum_i \Phi(\widetilde{n_2}\otimes z^{\natural}\otimes u_i )\otimes\widetilde{u_i}
+\sum_i u_i \otimes \Phi(\widetilde{u_i}\otimes x\otimes m_2 )\otimes\widetilde{n_1} \\
& =m_1 \otimes \sum_i \widetilde{n_2}\otimes z^{\sharp}\otimes u_i \otimes\widetilde{u_i}
+\sum_i u_i \otimes \widetilde{u_i}\otimes x\otimes m_2 \otimes\widetilde{n_1} \\
& =\,{}_A \la m_1\, , \, n_2 \ra \cdot z^{\natural}+x\cdot {}_A \la m_2 \, , \, n_1 \ra 
=(z\cdot{}_A \la n_2 \, , \, m_1 \ra )^{\natural}+x\cdot {}_A \la m_2 \, , \, n_1 \ra .
\end{align*}
Similarly,
\begin{align*}
& \widetilde{\Psi}(\widetilde{x^{\natural}}\otimes m_2 )\otimes \widetilde{n_1}
+m_1 \otimes\Psi(z\otimes n_2 )^{\widetilde{}} \\
& =\sum_i u_i \otimes \widetilde{\Phi}(\widetilde{u_i}\otimes\widetilde{x^{\natural}}\otimes m_2 )
\otimes\widetilde{n_1}+
m_1 \otimes\sum_i [u_i \otimes \Phi(\widetilde{u_i}\otimes z\otimes n_2 )]^{\widetilde{}} \\
& =\sum_i u_i \widetilde{\Phi}(\widetilde{u_i}\otimes\widetilde{x^{\natural}}\otimes m_2 )\otimes\widetilde{n_1}
+m_1 \otimes\sum_i \Phi(\widetilde{u_i}\otimes z\otimes n_2 )^{\widetilde{}}\otimes\widetilde{u_i} \\
& =\sum_i u_i \otimes\widetilde{\Phi}(\widetilde{u_i}\otimes\widetilde{x^{\natural}}\otimes m_2 ) \otimes\widetilde{n_1}
+ m_1 \otimes\sum_i \widetilde{\Phi}(\widetilde{n_2}\otimes\widetilde{z}\otimes u_i )\otimes \widetilde{u_i} .
\end{align*}
Also, since we identify $\widetilde{M}\otimes_A \widetilde{X}\otimes _A M$ with $\widetilde{Y}$
by the involutive Hilbert $B-B$-bimodule isomorphism $\widetilde{\Phi}$,
we see that
\begin{align*}& \sum_i u_i \otimes\widetilde{\Phi}(\widetilde{u_i}\otimes\widetilde{x^{\natural}}\otimes m_2 )
\otimes \widetilde{n_1}+m_1 \otimes \sum_i \widetilde{\Phi}(\widetilde{n_2}
\otimes\widetilde{z}\otimes u_i )\otimes\widetilde{u_i} \\
& =\sum_i {}_A \la u_i \, , \, u_i \ra\cdot \widetilde{x^{\natural}}\cdot{}_A \la m_2 \, , \, n_1 \ra
+\sum_i {}_A \la m_1 \, , \,  n_2 \ra \cdot \widetilde{z}\cdot {}_A \la u_i \, , \, u_i \ra \\
& =[(x\cdot {}_A \la m_2 \, , \, n_1 \ra )^{\natural}]^{\widetilde{}}
+(z\cdot {}_A \la n_2 \, , \, m_1 \ra )^{\widetilde{}} .
\end{align*}
Furthermore,
\begin{align*}
{}_A \la \widetilde{\Psi}(\widetilde{x^{\natural}}\otimes m_2 ) \, , \,
\widetilde{\Psi}(\widetilde{z^{\natural}}\otimes n_2 ) \ra +{}_A \la m_1 \, , \, n_1 \ra
& ={}_A \la \widetilde{x^{\natural}}\otimes m_2 \, , \, \widetilde{z^{\natural}}\otimes n_2 \ra
+{}_A \la m_1 \, , \, n_1 \ra \\
& ={}_A \la \widetilde{x^{\natural}}\cdot {}_A \la m_2 \, , \, n_2 \ra \, , \, \widetilde{z^{\natural}} \ra
+{}_A \la m_1 \, , \, n_1 \ra \\
& ={}_A \la [(x\cdot {}_A \la m_2 \, , \, n_2 \ra )^{\natural}]^{\widetilde{}} \, , \, \widetilde{z^{\natural}} \ra
+{}_A \la m_1 \, , \, n_1 \ra \\
& =\la (x\cdot {}_A \la m_2 \, , \, n_2 \ra )^{\natural} \, , z^{\natural} \ra_A
+{}_A \la m_1 \, , \, n_1 \ra \\
& ={}_A \la x\cdot {}_A \la m_2 \, , \, n_2 \ra \, , \, z \ra +{}_A \la m_1 \, , \, n_1 \ra \\
& ={}_A \la x\otimes m_2 \, , \, z\otimes n_2 \ra +{}_A \la m_1\, , \, n_1 \ra .
\end{align*}
Thus, the left $C_X$-valued inner products on $C_M$ and $C_M ' $ coincide.
Similarly we can see that the right $C_Y$-valued inner products
on $C_M$ and $C_M '$ coincide. Hence we obtain the following lemma:

\begin{lemma}\label{lem:bimodule} With the above notation, $C_M$ is a
$C_X - C_Y$-equivalence bimodule.
\end{lemma}
\begin{proof}
By the definitions of the left $C_X$-action and the left $C_X$-valued inner product
on $C_M$, we can see that Conditions (a)-(d) in \cite [Proposition 1.12]{KW1:bimodule}
hold. By the definitions of the right $C_Y$-action and the right $C_Y$-valued inner
product on $C_M$, we can also see that the similar conditions to Conditions
(a)-(d) in \cite [Proposition 1.12]{KW1:bimodule} hold. Furthermore, we can easily see
that the associativity of the left $C_X$-valued inner product and the right $C_Y$-valued inner product
holds. Since $M$ is an $A-B$-equivalence bimodule, there are finite subsets
$\{u_i \}_{i=1}^n$ and $\{v_j \}_{j=1}^m$ of $M$ such that
$$
\sum_{i=1}^n {}_A \la u_i , u_i \ra=1 , \quad
\sum_{j=1}^m \la v_j , v_j \ra_B =1 .
$$
Let $U_i =\begin{bmatrix} u_i & 0 \\
0 & u_i \end{bmatrix}$ for any $i$ and let $V_j =\begin{bmatrix} v_j & 0 \\
0 & v_j \end{bmatrix}$ for any $j$. 
Then $\{U_i\}$ and $\{V_j\}$ are finite subsets of $C_M$ and
$$
\sum_{i=1}^n {}_{C_X} \la U_i \, , \, U_i \ra =\sum_{i=1}^n \begin{bmatrix} u_i  & 0 \\
0 & u_i \end{bmatrix} \begin{bmatrix} \widetilde{u_i}  & 0 \\
0 & \widetilde{u_i} \end{bmatrix}
=\sum_{i=1}^n \begin{bmatrix}{}_A \la u_i \, , \, u_i \ra & 0 \\
0 & {}_A \la u_i \, , \, u_i \ra \end{bmatrix} =1_{C_X} .
$$
Similarly $\sum_{j=1}^m \la V_j \, , \, V_j \ra_{C_Y}=1_{C_Y}$.
Thus, since the associativity of the left $C_X$-valued inner product and  the right $C_Y$-valued inner 
product on $C_M$ holds, we can see that $\{U_i \}$ and $\{V_j \}$ are a right $C_Y$-basis
and a left $C_X$-basis of $C_M$, respectively. Hence by \cite [Proposition 1.12] {KW1:bimodule},
$C_M$ is a $C_X -C_Y$-equivalence bimodule.
\end{proof}

\begin{lemma}\label{lem:two-one} Let $A$ and $B$ be unital $C^*$-algebras.
Let $X$ and $Y$ be an involutive Hilbert $A-A$-bimodule and an involutive
Hilbert $B-B$-bimodule, respectively. Let $\mathcal{A}_X =\{A_t \}_{t\in \BZ_2}$ and
$\mathcal{A}_Y =\{B_t \}_{t\in \BZ_2}$ be $C^*$-algebraic bundles over $\BZ_2$
induced by $X$ and $Y$, respectively. We suppose that there is an $A-B$-equivalence bimodule $M$ such that
$$
Y\cong \widetilde{M}\otimes_A X\otimes_A M
$$
as involutive Hilbert $B-B$-bimodules. Then there is an $\mathcal{A}_X -\mathcal{A}_Y$-equivalence bundle $\mathcal{M}=\{M_t \}_{t\in \BZ_2}$
over $\BZ_2$ such that
$$
{}_{C} \la M_t \, , \, M_s \ra =A_{ts^{-1}} , \quad \la M_t \, , \, M_s \ra_{D}=B_{t^{-1}s}
$$
for any $t, s\in \BZ_2$, where $C=A\oplus X$ and $D=B\oplus Y$.
\end{lemma}
\begin{proof} Let $C_M$ be the $C_X -C_Y$-equivalence bimodule induced by $M$,
which is defined in the above. We identify $M\oplus (X\otimes_A M)$ with $C_M$ as vector
spaces over $\BC$ by the isomorphism defined by
$$
m_1 \oplus (x\otimes m_2 )\in M\oplus (X\otimes_A M)\mapsto \begin{bmatrix} m_1 & x\otimes m_2 \\
\widetilde{x^{\natural}}\otimes m_2 & m_1 \end{bmatrix}\in C_M .
$$
Since we identify $C=A\oplus X$ and $D=B\oplus Y$ with $C_X$ and $C_Y$, respectively,
$M\oplus (X\otimes_A M)$ is a $C-D$-equivalence bimodule by above identifications and Lemma \ref{lem:bimodule}.
Let $M_0 =M$ and $M_1 =X\otimes_A M$. We note that $X\otimes_A M$ is identified with
$M\otimes_B Y$ by the Hilbert $B-B$-bimodule isomorphism $\Psi$. Let $\mathcal{M}=\{M_t \}_{t\in \BZ_2}$.
Then by routine computations, $\mathcal{M}$ is an $\mathcal{A}_X -\mathcal{A}_Y$-equivalence
bundle over $\BZ_2$ such that
$$
{}_{C} \la M_t \, , \, M_s \ra =A_{ts^{-1}} , \quad \la M_t \, , \, M_s \ra_{D}=B_{t^{-1}s}
$$
for any $t, s\in \BZ_2$.
\end{proof}

\begin{prop}\label{prop:condition} Let $A$ and $B$ be unital $C^*$-algebras.
Let $X$ and $Y$ be an involutive Hilbert $A-A$-bimodule and an involutive
Hilbert $B-B$-bimodule, respectively. Let $\mathcal{A}_X =\{A_t \}_{t\in \BZ_2}$ and
$\mathcal{A}_Y =\{B_t \}_{t\in \BZ_2}$ be the $C^*$-algebraic bundles over $\BZ_2$
induced by $X$ and $Y$, respectively. Then the following conditions are equivalent:
\newline
$(1)$ There is an $\mathcal{A}_X -\mathcal{A}_Y$-equivalence bundle $\mathcal{M}=\{M_t \}_{t\in \BZ_2}$
over $\BZ_2$ such that
$$
{}_{C} \la M_t \, , \, M_s \ra =A_{ts^{-1}} , \quad \la M_t \, , \, M_s \ra_{D}=B_{t^{-1}s}
$$
for any $t, s\in \BZ_2$, where $C=A\oplus X$ and $D=B\oplus Y$.
\newline
$(2)$ There is an $A-B$-equivalence bimodule $M$ such that
$$
Y\cong \widetilde{M}\otimes_A X\otimes_A M
$$
as involutive Hilbert $B-B$-bimodules.
\end{prop}
\begin{proof}
This is immediate by Lemmas \ref{lem:one-two} and \ref{lem:two-one}.
\end{proof}

\begin{thm}\label{thm:module} Let $A$ and $B$ be unital $C^*$-algebras.
Let $X$ and $Y$ be an involutive Hilbert $A-A$-bimodule and an involutive
Hilbert $B-B$-bimodule, respectively. Let $A\subset C_X$ and $B\subset C_Y$ be the
unital inclusions of unital $C^*$-lgebras induced by $X$ and $Y$, respectively.
Then the following hold:
\newline
$(1)$ If there is an $A-B$-equivalence bimodule $M$ such that
$$
\widetilde{M}\otimes_A X\otimes_A M\cong Y
$$
as involutie Hilbert $B-B$-bimodules, then the unital inclusions $A\subset C_X$ and $B\subset C_Y$
are strongly Morita equivalent.
\newline
$(2)$ We suppose that $X$ and $Y$ are full with the both inner products and that $A' \cap C_X =\BC 1$. If
the unital inclusions $A\subset C_X$ and $B\subset C_Y$ are
strongly Morita equivalent,
then there is an
$A-B$-equivalence bimodule $M$ such that
$$
\widetilde{M}\otimes_A X\otimes_A M\cong Y
$$
as involutive Hilbert $B-B$-bimodules.
\end{thm}
\begin{proof} Let $\mathcal{A}_X =\{A_t \}_{t\in \BZ_2}$ and $\mathcal{A}_Y
=\{B_t \}_{t\in \BZ_2}$ be the $C^*$-algebraic bundles over $\BZ_2$ induced by
$X$ and $Y$, respectively. We prove (1). We suppose that there is an $A-B$-equivalence
bimodule $M$ such that
$$
\widetilde{M}\otimes_A X\otimes_A M\cong Y
$$ 
as involutive Hilbert $B-B$-bimodules. Then by Proposition \ref{prop:condition}, there is
an $\mathcal{A}_X -\mathcal{A}_Y$-equivalence bundle $\mathcal{M}=\{M_t \}_{t\in\BZ_2}$
over $\BZ_2$ such that
$$
{}_C \la M_t \, , \, M_s \ra =A_{ts^{-1}} \, , \quad \la M_t \, , \, M_s \ra_{D}=B_{t^{-1}s}
$$
for any $t, s\in \BZ_2$, where $C=A\oplus X$ and $D=B\oplus Y$. Hence by Proposition \ref{prop:easy},
the unital inclusions of unital $C^*$-algebras $A\subset C$ and $B\subset D$ are strongly Morita
equivalent. Since we identify $A\subset C$ and $B\subset D$ with $A\subset C_X$ and $B\subset C_Y$,
respectively, $A\subset C_X$ and $B\subset C_Y$ are strongly Morita equivalent. Next,
we prove (2). We suppose that $X$ and $Y$ are full with the both inner products and
that $A' \cap C_X =\BC 1$. Also, we suppose that $A\subset C_X$ and $B\subset C_Y$
are strongly Morita equivalent. Then $\mathcal{A}_X$ and $\mathcal{A}_Y$ are saturated
by Lemma \ref{lem:saturated}. Since the identity map $\id_{\BZ_2}$ is the only automorphism of $\BZ_2$,
by Theorem \ref{thm:inverse} there is an $\mathcal{A}_X -\mathcal{A}_Y$-equivalence bundle
$\mathcal{M}=\{M_t \}_{t\in \BZ_2}$ such that
$$
{}_C \la M_t \, , \, M_s \ra =A_{ts^{-1}}\, , \quad \la M_t \, , \, M_s \ra_D =B_{t^{-1}s}
$$
for any $t, s\in \BZ$. Hence Proposition \ref{prop:condition}, there is an $A-B$-equivalence bimodule $M$
such that
$$
Y\cong \widetilde{M}\otimes_A X\otimes_A M
$$
as involutive Hilbert $B-B$-bimodules.
\end{proof}

\end{document}